\documentclass[journal]{IEEEtran}
\ifCLASSINFOpdf
\else
\fi
\usepackage{amsmath}
\usepackage{graphicx}
\usepackage{subfigure}
\usepackage{amsmath}
\usepackage{algorithm}
\usepackage{algpseudocode}
\usepackage{cite}
\usepackage{multicol}
\usepackage{multirow}
\usepackage{mathtools}
\usepackage{amssymb}
\usepackage{amsthm}
\usepackage{color}
\usepackage{bm}
\usepackage{indentfirst}
\usepackage{url}
\usepackage{nomencl}
\usepackage{booktabs}
\usepackage{cite}
\usepackage[numbers,sort&compress]{natbib}\usepackage{amsmath}
\usepackage{graphicx}
\usepackage{subfigure}
\usepackage{amsmath}
\usepackage{algorithm}
\usepackage{algpseudocode}
\usepackage{cite}
\usepackage{multicol}
\usepackage{multirow}
\usepackage{mathtools}
\usepackage{amssymb}
\usepackage{color}
\usepackage{bm}
\usepackage{indentfirst}
\usepackage{url}
\usepackage{nomencl}
\usepackage{booktabs}
\usepackage{cite}
\usepackage[numbers,sort&compress]{natbib}

\newtheorem{theorem}{Theorem}
\newtheorem{lemma}{Lemma}
\newtheorem{proposition}{Proposition}
\newtheorem{corollary}{Corollary}
\newtheorem{definition}{Definition}
\newtheorem{remark}{Remark}

\begin{document}
%
\title{A Proximal Linearization-based Decentralized  Method for Nonconvex Problems with  Nonlinear Constraints}
%
%
%

\author{Yu~Yang,~\IEEEmembership{Student Member,~IEEE,}
	Guoqiang~Hu,~\IEEEmembership{Senior Member,~IEEE,}
	and~Costas~J.~Spanos,~\IEEEmembership{Fellow,~IEEE}
	\thanks{This  work  was  supported  by  the  Republic  of  Singapore’s  National  Research  Foundation  through  a  grant  to  the  Berkeley  Education  Alliance  for  Research  in  Singapore
		(BEARS)  for  the  Singapore-Berkeley  Building  Efficiency  and  Sustainability  in  the
		Tropics  (SinBerBEST)  Program.  BEARS  has  been  established  by  the  University  of  California,  Berkeley  as  a  center  for  intellectual  excellence  in  research  and  education  in
		Singapore.}
	\thanks{Yu Yang is with SinBerBEST, Berkeley Education 	Alliance for Research in Singapore, Singapore 138602 e-mail: (yu.yang@bears-berkeley.sg).}
	\thanks{Guoqiang Hu is with the School 	of Electrical and Electronic Engineering, Nanyang Technological University,
		Singapore, 639798 e-mail: (gqhu@ntu.edu.sg).}
	\thanks{Costas J. Spanos is with the Department of Electrical Engineering and 	Computer Sciences, University of California, Berkeley, CA, 94720 USA email: (spanos@berkeley.edu).}
}

\maketitle

\begin{abstract}

Decentralized optimization for non-convex problems are now demanding by many emerging applications (e.g., smart grids, smart building, etc.). Though dramatic progress has been achieved in convex problems,  the results for non-convex cases, especially with non-linear constraints, are  still largely unexplored. This is mainly due to the challenges imposed by the non-linearity and non-convexity,  which makes establishing the convergence conditions bewildered.
 This paper investigates decentralized optimization for  a class of structured non-convex problems characterized by: {\em{(i)}} nonconvex global objective function (possibly nonsmooth) and {\em{(ii)}}  coupled nonlinear constraints and local bounded convex constraints w.r.t. the agents. For such  problems,  a decentralized  approach called  {\em{Proximal Linearization-based Decentralized Method}} (PLDM) is proposed. 
 Different from the traditional (augmented) Lagrangian-based methods which usually require the exact (local) optima  at each iteration, the proposed method leverages a proximal linearization-based technique to update the decision variables iteratively,  which makes it computationally efficient and viable for the non-linear cases.   Under some standard conditions,  the PLDM's global convergence and local convergence rate to the $\epsilon$-critical points are studied based on the Kurdyka-Łojasiewicz property which holds for most analytical functions.  
Finally,  the performance  and efficacy of the method are  illustrated  through a numerical example and an application to multi-zone heating, ventilation and air-conditioning (HVAC) control. 


\end{abstract}

\begin{IEEEkeywords}
decentralized optimization, nonconvex problems,  coupled nonlinear constraints,  proximal linearization, augmented Lagrangian-based  method.
\end{IEEEkeywords}

%
\IEEEpeerreviewmaketitle

\section{Introduction}
This decade has seen the unprecedented computation demand arising from large-scale networked systems. 
Examples arising from smart grids and smart buildings include  electric vehicle charging management \cite{yang2017distributed, yang2018decentralized, yang2017stochastic, yang2016joint, long2021efficient}, peer-to-peer energy trading \cite{yang2022optimal, chen2022towards}, energy storage sharing \cite{yang2021optimal, yang2020selling},  and the multi-zone heating, ventilation and air-conditioning  (HVAC) control  \cite{radhakrishnan2016token, yang2020hvac, yang2021distributed, yang2021stochastic} etc. The past decades have witnessed the revived interests and dramatic progress  in decentralized optimization, especially for convex problems  (see for examples,  \cite{yuan2016regularized,  terelius2011decentralized, boyd2011distributed}). 
Nevertheless, the presence of  complex dynamic systems and big data (see \cite{ maharjan2016demand, cherukuri2018distributed} and references therein) nowadays requests for decentralized approaches  that  work for non-convex and non-linear context, which still remains an open question and has not been well established. 
This is mainly due to the intrinsic challenges imposed by the non-linearity and non-convexity, which lead to the NP-hard complexity as the existence of multiple optima, the lack of structure properties to guarantee the presence of optima (e.g., convexity and strong convexity), and the deficit of conditions to investigate convergence (e.g., sufficient global optimality conditions).

\subsection{Related Works}
In general, the available results  on (decentralized) non-convex optimization can be categorized based on the problem structures: {\em{(i)}} unconstrained non-convex problems, {\em{(ii)}} linearly constrained non-convex problem, and {\em{(iii)}} non-linearly constrained non-convex problems. 
Unless specified, we refer the constraints  above to the coupled constraints among different agents. 
For unconstrained problems, the most straightforward approaches  are  (sub-)gradient methods,  whose convergence has been  early established for convex problems (e.g., see \cite{bertsekas1997nonlinear}).  Recently, their extension \cite{khamaru2018convergence} and  variations, such as Frank-Wolf algorithm \cite{ lacoste2016convergence}, proximal gradient method \cite{ zeng2018nonconvex},  have been discussed for non-convex cases. 
The broad results are that some local optima (i.e., critical points) can  be approached with diminishing step-sizes.
Generally, the above works are mainly focused on investigating the convergence of  such methods for non-convex situations rather than achieving decentralized computation.  
On that basis, some decentralized paradigms by combining (proximal) gradient-based methods with alternating minimization technique (i.e, Gauss-Seidel update) have been proposed, which include proximal alternating linearized minimization (PALM)  \cite{bolte2014proximal, xu2013block}, inexact proximal gradient methods (IPG) \cite{gu2018inexact} and some variations \cite{nikolova2018alternating, attouch2010proximal}.

As the noteworthy performance of the methods of multipliers (MMs) or augmented Lagrangian-based methods, such as  alternating direction multiplier method  (ADMM) \cite{boyd2011distributed},  has been  thoroughly observed and  understood in tackling coupled (linear) constraints in convex context,  their extension to linearly coupled non-convex problems seems natural.  
The recent years have witnessed the widespread discussions on decentralized optimization for non-convex problems subject to coupled linear constraints under the Lagrangian-based framework, especially ADMM and its variations (see \cite{davis2016convergence, davis2017faster, he20121} and the references therein).    A comprehensive survey of ADMM and its variants for constrained optimization is available \cite{yang2022survey}. 
In principle,  ADMM  can leverage the fast convergence feature of augmented Lagrangian methods and the separable structure of dual decomposition  based on the alternating minimization technique (i.e, Gauss-Seidel update). 
 Nevertheless,  in contrast to convex cases (see \cite{davis2016convergence, davis2017faster, he20121}),   it's not straightforward or trivial  to  achieve their performance guarantee in non-convex context. 
 Though favorable performance of ADMM and its variations does have  been observed and reported  in various applications involving non-convexity (see for examples, \cite{erseghe2014distributed, xu2012alternating, yang2022proximal}),  
 the theoretical understandings are still  fairly limited and deficient except for \cite{hong2016convergence, magnusson2015convergence, chatzipanagiotis2017convergence,  wang2018convergence}, where some results have been established for special structured problems.
 
With the  prevalence of complex multi-agent dynamic systems, growing demand has been raised for decentralized non-convex optimization with coupled non-linear constraints. Such examples  include  the optimal power flow (OPF) control  \cite{liu2017real},  and the multi-zone heating, ventilation and air-conditioning  (HVAC) control in buildings \cite{radhakrishnan2016token, yang2020hvac, yang2021distributed, yang2021stochastic}, etc.
The non-linear constraints generally arise from the coupled complex system dynamics, which are pivotal while designing decentralized controllers.
 However, such non-linear couplings  further compound the difficulties and challenges to develop decentralized computing paradigms.  
  The complexities mainly stem  from  \emph{i)} the lack of standard framework to deal with such general non-linear constraints;
  \emph{ii)} the difficulties  to ensure feasibility  of  the coupled nonlinear constraints while performing decentralized computing; 
  \emph{iii)} the challenges to investigate the convergence of a specific decentralized algorithm without any structure properties (e.g., convexity, strong convexity, etc).
  Though the augmented Lagrangian methods are appealing in deal with coupled constraints, their extension to  general non-linear constraints are not straightforward or well-founded.  Moreover, the existing standard augmented  Lagrangian-based framework generally require the exact optimization (e.g.. local optima) at each iteration, which are not viable for nonlinear and nonconvex cases practically. 
Surprisingly,  though lack of theoretical foundations, the favorable performance  of  some Lagrangian-based methods  also has been  observed  in some applications (e.g. matrix completion and factorization \cite{xu2012alternating}, optimal power flow  \cite{erseghe2014distributed, magnusson2015distributed}). This has inspired some recent exploratory studies on their theoretical understandings \cite{Giesen2018DistributedCO, wang2017nonconvex, sun2019two}.  For example,  \cite{Giesen2018DistributedCO} studied the direct extension of ADMM to two-block convex problems with coupled non-linear but separable constraints. 
\cite{shi2017penalty} proposed a tailored penalty dual decomposition (PPD) method by combining penalty method and augmented Lagrangian method to tackle  non-convex problems with non-linear constraints.  
Except those, there exist another two excellent recent works that have  shed some light  on such situations.  One is \cite{sun2019two} which \emph{i)} thoroughly and systematically investigated the intrinsic challenges to establish convergence guarantee for ADMM in such situations; and \emph{ii)} resorted to a two-level nested  framework  as a remedy.  However, \cite{sun2019two} generally requires the joint optimization of  multi-block non-linear problem at each iteration in the inner loop, which are not viable or attainable in practice.
Another noteworthy work is \cite{bolte2018nonconvex} which investigated the general conditions for  Lagrangian-based framework to  achieve global convergence guarantee in non-convex context. Rather than proposing a specific algorithm,  \cite{bolte2018nonconvex} seems more focused on  establish a general framework and leaves the algorithm design open.

Overall, two key points from the status quo that may necessitate our attention.
\emph{First},  the above recent progress on non-convex optimization are mainly attributed to the establishment of 
Kurdyka-Łojasiewicz properties hold by many analytical functions, which was first proposed in  \cite{kurdyka1998gradients, lojasiewicz1963propriete} and later extended in  \cite{bolte2007lojasiewicz, bolte2010characterizations}.
The KL properties  are powerful as they make it possible to  characterize sequence around critical points without convexity.  \emph{Second},  the key ideas to investigate the methods' convergence are mainly twofolds:  
\emph{i)} studying the convergence of primal and dual sequences by inspecting a tailored 
Lyapunov function; 
\emph{ii)} investigating the local convergence of the algorithm based on the KL properties. 
However, a typical and difficult problem  is that the augmented Lagrangian function generally 
oscillates in the non-convex context, which  makes the use and  design of Lyapunov functions particularly difficult \cite{bolte2018nonconvex}.

 \subsection{Our Contributions}
 Motivated by the recent progress on non-convex optimization, this paper seeks to investigate  decentralized optimization for a class of structured problems with \emph{(i)}  nonconvex global objective function (possibly nonsmooth) and \emph{(ii)} local bounded convex constraints and
 coupled nonlinear constraints w.r.t. the agents. 
 
In general, to develop a viable decentralized method, we need to overcome two main challenges.
\emph{First}, we need to realize the difficulties of calculating local optima of non-linear problems at each iteration as required by most existing decentralized paradigms \cite{shi2017penalty}.
\emph{Second}, we need to figure out the convergence conditions to achieve performance guarantee.  To address such issues, this paper proposes a  \emph{ Proximal Linearization-based  Decentralized Method}  (PLDM). The main ideas are twofolds.  \emph{First}, considering the difficulties to ensure the feasibility of the non-linear coupled constraints,  we first introduce some  consensus variables to eliminate the nonlinear couplings. \emph{Second},  to overcome the intrinsic  challenges to guarantee convergence as explained in \cite{sun2019two}, we solve a relaxed non-convex problems  in a decentralized manner by  combing the augmented Lagrangian-based framework, the \emph{alternating minimization} (i.e., i.e, Gauss-Seidel), and \emph{proximal linearization} \cite{bolte2014proximal, bolte2019proximal} to approximate the solutions of  the original problem.
In particular, different from the
 traditional MMs and  augmented Lagrangian-based methods where the exact local optima of the problems are required at each iteration, the proposed method leverages proximal linearization-based technique to update decision variables at each iteration,  which makes  it computationally efficient and viable for the nonconvex and nonlinear cases.  The main contributions of this paper are outlined, i.e.,

 \begin{itemize}
 	\item   We propose  a  PLDM  for a class of structured nonconvex problems subject to  coupled nonlinear constraints and  local bounded convex constraints.
 	
 	\item  The global convergence and local convergence rate  of the method to the $\epsilon$-critical points are studied by inspecting a tailored Lyapunov function and the Łojasiewicz property of the AL function.
 	\item  The performance of the decentralized method is illustrated by presenting a numeric example and an application to multi-zone heating, ventilation and air-conditioning (HVAC) control. 
 	 
 \end{itemize}
 
 The remainder of this paper is structured. 
 In Section II, the main notations and  the problem are presented.
 In Section III, the PLDM  is introduced.
 In Section IV, the global convergence and local  convergence rate  of the method are  investigated.
 In Section IV, the performance of the method is evaluated through a numeric example and an application.
 In Section V, we briefly conclude and discuss this paper.

\section{The Problem and Main Notations}

\subsection{Notations}
Throughout this paper, we use $\mathbb{N}$, $\mathbb{R}$,  $\mathbb{R}^n$ ($\mathbb{R}^n_{+}$),   and $\mathbb{R}^{m\times n}$ to  denote the spaces of integers,  reals, $n$-dimensional (positive) real vectors,  and $m\times n$-dimensional real matrices, respectively.
The superscript $T$ denotes the transpose operator. 
 $\bm{I}_n$  and $\bm{O}_n $ denote the $n$-dimensional identify and zero matrices. 
Without specification, $\Vert \cdot \Vert $ denotes $\ell_2$ norm.  $P_{\mathcal{X}} [\cdot]$ represents the projection operation on the set $\mathcal{X}$. We use $\nabla f=\big( \frac{\partial f}{ \partial x_1}, \frac{\partial f}{ \partial x_2}, \cdots, \frac{\partial f}{ \partial x_n}\big)^T$ to denote the gradients of $f: \mathbb{R}^n \rightarrow \mathbb{R}$ w.r.t its entities.
If $\bm{h}: \mathbb{R}^n \rightarrow \mathbb{R}^m$, i.e.,  $\bm{h}(\bm{x})=(h_1(\bm{x}), h_2(\bm{x}), \cdots, h_m(\bm{x}))^T$ with $h_i: \mathbb{R}^n \rightarrow \mathbb{R}$, we have $\nabla \bm{h}=\big( \nabla h_1, \nabla h_2, \cdots, \nabla h_m \big)^T$. 
$\prod_{i=1}^N \mathcal{X}_i$  denotes the Cartesian product of the sets $\mathcal{X}_i$ ($i=1, 2, \cdots, N$).  
The other  notations are  standard and follow the literature  \cite{shi2017penalty, zeng2018nonconvex}.

\subsection{The Problem}
This paper focuses on a class of problems given by
\begin{equation} \label{main problem}
\begin{split}
 \min_{\bm{x}_i, i=1, 2, \cdots, N} &\sum_{i=1}^N  f_i(\bm{x}_i)+\sum_{i=1}^N \phi_i (\{x_j\}_{j \in \mathcal{N}_i}) \\
 s. t. & \quad  h_i(\{x_j\}_{j \in \mathcal{N}_i})=0, \quad  ~\forall i \in \mathcal{N}. \\
 &\quad  \bm{x}_i \in \bm{\mathcal{X}}_i,  \quad \quad \quad \quad \quad~\forall i \in \mathcal{N}.
\end{split}
\end{equation}
where $i$ denotes the index of the agents from the set $\mathcal{N}=\{1, 2, \cdots, N\}$.  Here $\bm{x}_i \in \mathbb{R}^{n_i}$ denotes the local decision component of Agent  $i$.  $\mathcal{N}_i$  is alluded to the collection of agent $i$ and its neighbors.  

Note that problem  \eqref{main problem} has global objective function which is composed by the separable parts $f_i:\mathbb{R}^{n_i} \rightarrow  \mathbb{R}$ and the composite parts $\phi_i: \mathbb{R}^{\bar{n}_i} \rightarrow  \mathbb{R}$ ($\bar{n}_i=\sum_{i \in \mathcal{N}_i} n_i$) w.r.t. agents. Wherein the objective function $f_i$ and $\phi_i$ may be  nonconvex (possibly nonsmooth). 
 $h_i: \mathbb{R}^{ \bar{n}_i} \rightarrow  \mathbb{R}^{m_i}$ denotes the coupled nonlinear constraints pertaining to Agent $i$, which are  smooth and differentiable.  As inequality constraints could be transformed to equality constraints by introducing slack variables, this paper only investigate equality constraints.
Besides, there exist  local bounded convex constraints  represented by $\mathcal{X}_i$ ($i \in \mathcal{N} $) for the agents.  In addition, we make the following assumptions for problem \eqref{main problem} in our analysis, i.e., 

\begin{enumerate}
	\item[{\em{(A1)}}] The equality constraints $h_i$  are continuously differentiable  over $\mathcal{X}_i$ ($h_i$ and $\nabla h_i$ are Lipschitz continuous  with constants $L_{h_i}$ and $M_{h_i}$).
	\item[{\em{(A2)}}] The separable parts $f_i$ and $\nabla f_i$ are   Lipschitz continuous  with constants 	$L_{f_i}$ and $M_{f_i}$ over $\mathcal{X}_i$, respectively. 
	\item[{\em{(A3)}}] The coupled parts $\phi_i$  and $\nabla \phi_i$ are Lipschitz continuous  with constants $L_{\phi_i}$ and $M_{\phi_i}$ over $\prod_{j \in \mathcal{N}_i}\mathcal{X}_i$, respectively. 
\end{enumerate}

\subsection{Problem Reformulation}
As aforementioned, it's generally challenging to tackle problem \eqref{main problem} with non-linearity and non-convexity both in the objective and the constraints with performance guarantee. Therefore,  this part  presents  some reformulations of  the problem as a necessary preparation for the following study. 
First, to handle the coupled nonlinear couplings, we  introduce a block of consensus variables 
$\bm{Z}=((z_1)^T, (z_2)^T, \cdots, (z_N)^T)^T \in \mathbb{R}^{\sum_{i \in \mathcal{N}} n_i}$, which represents the hypothetical copy of  the collected decision components for all the agents. 
As displayed in Fig. \ref{Reformulated_variable},  each Agent $i$ will hold a local copy of the augmented decision variables denoted by  $\bm{X}_i\!=\!(x_i, \{x^j_i \}_{j \in \mathcal{N}_i \setminus \{i\}})^T \!\in\!\! \mathbb{R}^{\bar{n}_i}$ ($\bar{n}_i=\sum_{j \in \mathcal{N}_i } n_j$), in which  $x^j_i$ denotes the local copy of the decision component for its neighboring Agent $j$.  Meanwhile, we assume there is a virtual  Agent $0$, who is obligated to manage the block of consensus decision variables $\bm{Z}$. 
Intuitively, for an algorithm to converge, we require 
\begin{subequations}
 \begin{alignat}{2}
&x_i=z_i, ~~x^j_i=z_j. \label{eq:2a} \\
&x^i_j=z_i, ~~ x_j=z_j.  \label{eq:2b}
 \end{alignat}
\end{subequations}
\vspace{-0.4 in}
\begin{figure}[h] 
	\centering
	\includegraphics[width=2.2 in]{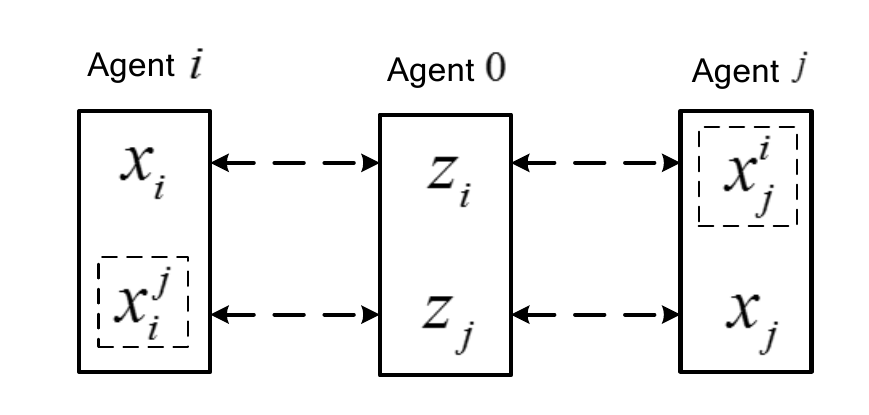}\\
	\caption{The local and global (consensus)  decision variables.} \label{Reformulated_variable} 
\end{figure} 

In this case,  problem (\ref{main problem}) can be restated as 
\begin{subequations}
	\begin{alignat}{4}
	\label{P3} \min_{\bm{X}_i, i=1, 2, \cdots, N.   \bm{Z}}&  \quad  \sum_{i=1}^N  \tilde{f}_i (\bm{X}_i)+\sum_{i=1}^N \phi_i  ( \bm{X}_i )  \tag{3}\\
	\label{eq:3a} s. t. \quad & h_i(\bm{X}_i)=\bm{0}, ~\forall i \in \mathcal{N}. \\
	\label{eq:3b} &\bm{X}_i\!=\! \bm{E}_i  \bm{Z}, ~~~\forall i \in \mathcal{N}. \\
 \label{eq:3c}	&\bm{Z} \in \bm{\mathcal{X}}. 
	\end{alignat}
\end{subequations}
where we have  $\bm{\mathcal{X}}=\prod_{i \in \mathcal{N}} \mathcal{X}_i$.
$\tilde{f}_i: \mathbb{R}^{\bar{n}_i} \rightarrow  \mathbb{R}$ denotes the extended function of $f_i$. 
$\bm{E}_i=diag\{\bm{I}_{n_j}\}_{j \in \mathcal{N}_i}
\in \mathbb{R}^{\bar{n}_i   \times \sum\limits_{j  \in \mathcal{N }} n_j }$
are constant matrices. \eqref{eq:3b} denotes the compact form of the equality constraints \eqref{eq:2a} and \eqref{eq:2b}.

Observe \eqref{P3}, we note that the objective function is now decomposable w.r.t. the agents.  However, there appear three types of constraints: {\em{(i)}} local  nonlinear and nonconvex constraints \eqref{eq:3a}, {\em{(ii)}} coupled linear constraints \eqref{eq:3b}, and {\em{(iii)}} local bounded convex constraints \eqref{eq:3c}. There exist two blocks of decision variables, i.e., the primal and consensus decision variables  $\bm{X}\!=\![\bm{X}^T_i]_{i \in \mathcal{N}}^T$ and $\bm{Z}$. 
Nevertheless, we note that the local constraints are now not regular as $[(\nabla_{\bm{X}_i} h_i (\bm{X}_i))^T, \bm{I}_{\bar{n}_i} ]^T$  is not full row rank.
In fact, this presents one of the intrinsic challenges to investigate the convergence of a general decentralized methods for such problems (this may be  comprehended in the rigorous analysis later). 
Here this may be understood that  even a particular (decentralized) algorithm is able to force $\bm{Z}$ to some (local) optima $\bm{Z}^*$,  no $\bm{X}^*$ that satisfies the constraints \eqref{eq:3a} and \eqref{eq:3c} simultaneously can be found.  To address such a challenging issue, we  introduce some slack  variables  $\bm{Y}_i \in \mathbb{R}^{\bar{n}_i}$ ($i \in \mathcal{N}$) and  restate problem \eqref{P3}  as
\begin{subequations}
	\begin{alignat}{4} 
	\label{P4} \min_{\bm{\bar{X}}_i, i \in \mathcal{N}.  \bm{Z}}&  \quad  \sum_{i=1}^N  \tilde{f}_i (\bm{X}_i)+\sum_{i=1}^N \phi_i  ( \bm{X}_i ) \tag{4}\\
	\label{eq:4a} s. t. \quad & h_i(\bm{X}_i)=\bm{0}, ~~~~\forall i \in \mathcal{N}. \\
	&\label{eq:4b} \bm{A}_i \bm{\bar{X}}_i=\! \bm{E}_i  \bm{Z}, ~~\forall i \in \mathcal{N}.  \\
	&\label{eq:4d} \bm{Y}_i=\bm{0},  ~~~~~~~~~~\forall i \in \mathcal{N}. \\
	&\label{eq:4c} \bm{Z} \in \bm{\mathcal{X}},  
	\end{alignat}
\end{subequations}
where $\bm{A}_i=[\bm{I}_{\bar{n}_i}, \bm{I}_{\bar{n}_i}] \in \mathbb{R}^{\bar{n}_i \times 2 \bar{n}_i }$. $\bm{\bar{X}}_i=\big( (\bm{X}_i)^T, (\bm{Y}_i)^T \big)^T$ ($i\in \mathcal{N}$) denotes the extended local decision variable  hold by  Agent $i$. 

Note that problem \eqref{P3} is equivalent to \eqref{P4} as the slack variables $\bm{Y}_i$ are forced to be \emph{zero}  in the constraints. 
However,  the local constraints w.r.t the extended local decision variables $\bm{\bar{X}}_i$ ($i \in \mathcal{N}$) are still not regular, therefore we have to resort to the following relaxed  problem: 
\begin{subequations}
	\begin{alignat}{4} 
\label{P5} \min_{\bm{\bar{X}}_i, i \in \mathcal{N}.  \bm{Z}}&  \quad  \sum_{i=1}^N  \tilde{f}_i (\bm{X}_i)+\sum_{i=1}^N \phi_i  ( \bm{X}_i )+\sum_{i=1}^N M_i \Vert \bm{Y}_i \Vert^2 \tag{5}\\
\label{eq:5a} s. t. \quad & h_i(\bm{X}_i)=\bm{0}, ~~~~\forall i \in \mathcal{N}. \\
&\label{eq:5b} \bm{A}_i \bm{\bar{X}}_i=\! \bm{E}_i  \bm{Z}, ~~\forall i \in \mathcal{N}.  \\
&\label{eq:5c} \bm{Z} \in \bm{\mathcal{X}},  
	\end{alignat}
\end{subequations}
where $M_i > 0$  ($i\in \mathcal{N}$) denotes some positive penalty parameter.  In contrast of   \eqref{P5} with \eqref{P4},  one may note that  the constraints $\bm{Y_i}=\bm{0}$ for the slack variables have been softened by appending some penalty terms  $M_i \Vert \bm{Y}_i \Vert^2$  in the global objective function.  As \emph{i)} problem \eqref{P5} contains all the feasible solutions of \eqref{P4}; \emph{ii)} $P^{(5), *} \leq P^{ (4), *}$ ($P^{(\cdot), *}$ denotes the optima of the problems),  problem \eqref{P5} can be regarded as an relaxation of problem \eqref{P4}. 
Moreover,  we note that with $M_i \rightarrow 0$  ($i\in \mathcal{N}$),  the well-posedness optimal solutions of \eqref{P5} (with bounded objective value) are exactly those for the original problem \eqref{P4} as we will have $\bm{Y}^*_i \rightarrow \bm{0}$ ($\forall i \in \mathcal{N}$) otherwise $P^{(5), *} \rightarrow \infty$.

The main ideas of this paper are twofolds.  \emph{First},  we  investigate decentralized method for problem \eqref{P5} with performance guarantee. 
\emph{Second},  we prove that the attained solutions are  the $\epsilon$-critical points of \eqref{P4} under some mild conditions. 
	Before that, we first make the following  extra assumptions.
\begin{itemize}
	\item [{\em{(A4)}}]  $M_i > 0$ ($i \in \mathcal{N}$) are sufficiently large.
\end{itemize}

\begin{remark}
  Generally, rigorous analysis requires  $M_i \rightarrow +\infty$ to guarantee the equivalence of problem \eqref{P4} and \eqref{P5}. However,  a sufficiently large positive value is enough to guarantee the sub-optimality in practice. 
\end{remark}

\begin{itemize}
	\item [{\em{(A5)}}]  $\bm{F}_i(\bm{\bar{X}}_i)
	\!\!=\!\!
	\left(  
	\begin{array}{ccc}
	\nabla_{\bm{X}_i} h_i(\bm{X}_i) &\!\!\!\!& \bm{O}_{n_i} \\
	&\bm{A}_i&\\
	\end{array} 
	\right)$
	 is uniformly regular with constant $\theta$   over the bounded set   $\bm{\mathcal{\bar{X}}}^{\eta}_i$ (see \textbf{Definition} 3),  where $\bm{\mathcal{\bar{X}}}_i^{\eta}=P_{\mathbb{R}^{\bar{n}_i}}[\bm{\mathcal{X}}^{\eta}] $, with  $\bm{\mathcal{X}}^{\eta}=\{\bar{\bm{X}} \in \mathbb{R}^{\sum_{i \in \mathcal{N}} 2 \overline{n}_i}  $ $|  \sum_{i=1}^N \Vert h_i(\bm{X}_i) \Vert \leq \eta \}$ ($\eta$ is a positive threshold).
	
\end{itemize}

\begin{remark}
(A5)  is standard  and can generally  be satisfied through regulating the dimension  of the slack variables. 
\end{remark}



 \section{Proximal Linearization-Based Decentralized Method}
This section first presents a decentralized method called PLDM  for problem \eqref{P5} and gives the main results on its convergence.
After that  we prove that the attained solutions are  the $\epsilon$-critical points of the original problem \eqref{P4}. In particular, the proposed method is established under the
general augmented Lagrangian-based framework for nonconvex context (see \cite{hours2014augmented, bolte2018nonconvex} for examples ) and contains the three standard steps: {\em{(i)}} {\em{primal update}} based on {\em{proximal linearization}}, {\em{(ii)}} {\em{dual update}} and {\em{(iii)}} {\em{adaptive step}} for penalty factors.


To handle the local nonlinear constraints \eqref{eq:4a} and the coupled linear constraints \eqref{eq:4b}, we define the Augmented Lagrangian (AL) function as
\begin{eqnarray} \label{Lagrangian function}
\begin{split}
\mathbb{L}_{\rho}& (\bm{\bar{X}}, \bm{\lambda}, \bm{\mu})=\sum_{i=1}^N \tilde{f}_i(\bm{X}_i) +\sum_{i=1}^N \phi_i (\bm{X}_i )+\sum_{i=1}^N M_i \Vert \bm{Y}_i\Vert^2 \\
&+\sum_{i=1}^N (\bm{\lambda}_i)^T  h_i(\bm{X}_i)+\sum_{i=1}^N \frac{\rho}{2} \big \Vert  h_i(\bm{X}_i)\big \Vert^2\\
&+\sum_{i=1}^N  \big( \bm{\mu}_i  \big)^T\!( \bm{A}_i \bm{\bar{X}}_i-\bm{E}_i  \bm{Z} )\!+\!\sum_{i=1}^N \! \frac{\rho}{2} \big \Vert  \bm{A}_i \bm{\bar{X}}_i-\bm{E}_i \bm{Z} \big \Vert^2
\end{split}
\end{eqnarray}
where $\bm{\bar{X}}=[\bm{\bar{X}}_i]_{i \in \mathcal{N}} \in \mathbb{R}^{\sum_{i \in \mathcal{N} }2\bar{n}_i}$ denotes the augmented decision variable for the problem.  $\bm{\lambda}=[\bm{\lambda}_i]_{i \in \mathcal{N}} \in \mathbb{R}^{\sum_{i \in \mathcal{N}} m_i}$ and $\bm{\mu}=[\bm{\mu}_i]_{i \in \mathcal{N}}  \in \mathbb{R}^{\sum_{i \in \mathcal{N}} 2 \bar{n}_i } $ are Lagrangian multipliers. $\rho>0$ denotes the penalty factor.

For given Lagrangian multipliers $\bm{\lambda}$ and $\bm{\mu}$,  the primal problem needs to be solved is  given by
\begin{equation} \label{eq:pp}
\begin{split}
\min_{\bm{\bar{X}}, \bm{Z}} ~~&\mathbb{L}_{\rho} (\bm{\bar{X}},  \bm{\lambda}, \bm{\mu}) \\
 s. t. ~~&\bm{Z}  \in  \bm{\mathcal{X}}. \\
\end{split}
\end{equation}

{\em{Primal Update}}: observe that \eqref{eq:pp} has two blocks of decision variables (i.e.,$ \bm{\bar{X}}$ and $\bm{Z}$). For general MMs, a joint optimization of  the two decision variable blocks is usually required \cite{dhingra2018proximal}. However,  the primal  problem \eqref{eq:pp} is  nonlinear and nonconvex, even obtaining a local optima is difficult with existing approaches. To handle this, we achieve {\em{primal update}} at each iteration in a decentralized manner by performing two steps : {\em{(i)}}  \emph{proximal linearization} of the nonlinear subproblems and {\em{(ii)}} Gauss-seidal update w.r.t. the decision variable blocks (i.e., $\bm{\bar{X}}$, $\bm{Z}$). Observing that AL function (\ref{Lagrangian function})  is decomposable w.r.t.  the agents with the extended local decision variable $\bm{\bar{X}}_i$ ($\forall i \in \mathcal{N}$), we define the subproblems for each Agent $i$ as 
{\small
\begin{equation} \label{local obj for node i}
\begin{split}
& \underset{\bm{\bar{X}}_i}{\operatorname{min}} \quad \mathbb{L}^i_{\rho} (\bm{\bar{X}}_i, \bm{Z}, \bm{\lambda}_i, \bm{\mu}_i)=\tilde{f}_i(\bm{X}_i)+\phi_i (\bm{X}_i)+M_i \Vert \bm{Y}_i\Vert^2\\
&\quad +(\bm{\lambda_i})^T h_i(\bm{X}_i)+\frac{\rho}{2} \big\Vert  h_i(\bm{X}_i)\big\Vert^2+ \big(\bm{\mu}_i\big)^T ( \bm{A}_i \bm{\bar{X}}_i-\bm{E}_i \bm{Z})\\
&\quad + \frac{\rho}{2}\big \Vert \bm{A}_i \bm{\bar{X}}_i- \bm{E}_i \bm{Z} \big \Vert^2, ~~\forall i \in \mathcal{N}. 
\end{split}
\end{equation} }
\normalsize

By performing {\emph{proximal linearization}} on the nonlinear and nonconvex subproblems \eqref{local obj for node i} at each iteration $k$, we have
\small
\begin{displaymath} 
\begin{split}
&\min_{\bm{\bar{X}}_i} \mathbb{\tilde{L}}^i_{\rho_{k}} (\bm{\bar{X}}_i, \bm{\bar{X}}^k_i, \bm{Z}^k, \bm{\lambda}^k_i, \bm{\mu}_i^{k})= (\bm{\mu}^{k}_i)^T ( \bm{A}_i \bm{\bar{X}}_i-\bm{E}_i \bm{Z}^{k+1})\\
&+\! \frac{\rho_{k}}{2} \Vert \bm{A}_i \bm{\bar{X}}_i\!-\!\bm{E}_i \bm{Z}^{k+1}\Vert^2\!+\!\langle \nabla_{\bm{X}_i} g_i (\bm{\bar{X}}^k_i, \bm{\lambda}_i^k, \rho_{k}), \bm{\bar{X}}_i\!-\!\bm{\bar{X}}^k_i \rangle\\
&+\frac{c^k_i}{2} \Vert  \bm{\bar{X}}_i-\bm{\bar{X}}^k_i\Vert^2 , ~~\forall i \in \mathcal{N}.
\end{split}
\end{displaymath}	
\normalsize
where $c^k_i$ denotes the step-size for subproblem $i$ at iteration $k$. Besides, we define
$g_i(\bm{\bar{X}}_i, \bm{\lambda}_i, \rho)\!=\!\tilde{f}_i(\bm{X}_i)\!+\!\phi_i(\bm{X}_i)\!+\!(\bm{\lambda_i})^Th_i(\bm{X}_i)+\frac{\rho}{2} \Vert h_i(\bm{X}_i)\Vert^2+M_i\Vert \bm{Y}_i \Vert^2$.

\begin{remark}
	If the objective function $\tilde{f}_i$ or $\phi_i$ is non-smooth w.r.t the local decision variable $\bm{\bar{X}}_i$, we can remove it from $g_i(\bm{\bar{X}}^k_i, \bm{\lambda}^k_i, \rho_{k})$ and keep them in $\mathbb{\tilde{L}}^i_{\rho_{k}} (\bm{\bar{X}}_i, \bm{\bar{X}}^k_i, \bm{Z}^k, \bm{\lambda}^k_i, \bm{\mu}^k_i)$.
\end{remark}

For the consensus variable $\bm{Z}$, the subproblem is given by
\begin{equation} \label{eq:subproblem8}
\begin{split}
&\underset{\bm{Z}}{\operatorname{min}} \quad \mathbb{L}_{\rho}(\bm{Z}, \bm{\bar{X}}, \bm{\mu})\!=\!\sum_{i=1}^N \big( \bm{\mu}_i \big)^T \big(\bm{A}_i \bm{\bar{X}}_i\!-\! \bm{E}_i \bm{Z}\big) \\
&\quad +\sum_{i=1}^N \frac{\rho }{2} \big \Vert \bm{A}_i\bm{\bar{X}}_i-\!\bm{E}_i \bm{Z} \big \Vert^2, ~~s. t. ~~\bm{Z}  \in  ~{\bm{\mathcal{X}}}. 
\end{split}
\end{equation}
We note that  subproblem \eqref{eq:subproblem8} is a quadratic programming (QP),  which can be solved efficiently.

{\em{Dual update \& Adaptive step}} : the dual variables  are updated following the standard augmented Lagrangian methods~(see \cite{chatzipanagiotis2015augmented} for example). However,  we introduce an {\em{adaptive step}} to dynamically update the penalty factor $\rho$.  In particular, we pre-define a sub-feasible region  regarding  the non-linear constraints of problem \eqref{P5}, i.e., 
$$\bm{\mathcal{\bar{X}}}^{\eta} \triangleq \big\{ \bar{\bm{X}} \in \mathbb{R}^{\sum_{i \in \mathcal{N}} 2 \overline{n}_i} |  \sum_{i=1}^N \Vert h_i(\bm{X}_i) \Vert \leq \eta  \big\}. $$
We iteratively increase $\rho_k$ with an increment $\delta$ until $\bm{\bar{X}}^k$ is forced into the pre-defined sub-feasible region $\bm{\mathcal{\bar{X}}}^{\eta}$.

The main steps of the proposed PLDM method for solving problem \eqref{P5} are summarized in \textbf{Algorithm} \ref{CADMM}. The algorithm starts by initializing the Lagrangian multipliers, penalty factor, and decision variables.  Afterwards, the main steps include the alternative update of the two decision  blocks ($\bm{\bar{X}}$, $\bm{Z}$) ({Step} 3-5), the Lagragnain multipliers ($\bm{\lambda}$, $\bm{\mu}$) (Step 6), and  the penalty factor ($\rho$) (Step 7). We note that as the primal problem \eqref{eq:pp} is  decomposable  w.r.t  the agents,   the update of the primal decision variable block ($\bm{\bar{X}}$) can be performed in parallel by the agents.  
In \textbf{Algorithm} \ref{CADMM},  the stopping criterion is defined as the residual error  bound of the constraints, i.e.,  
\begin{displaymath}
\begin{split}
R(k)=\sum_{i=1}^N \Big\{\Vert h_i(\bm{X}^{k+1}_i)\Vert+\Vert \bm{A}_i \bm{\bar{X}}^{k+1}_i-\bm{E}_i  \bm{Z}^{k+1} \Vert \Big\} \leq \epsilon
\end{split}
\end{displaymath}
where $\epsilon$ is a constant threshold. 

The algorithm iterates until the stopping criterion is reached. Still, this does not mean  the  convergence of the algorithm. This  needs to be studied in greater detail later.


\begin{algorithm}
	\caption{Proximal Linearization-based Decentralized Method (PLDM) for Nonconvex and Nonlinear Problems} \label{CADMM}
	\begin{algorithmic}[1]
		\State  \textbf{Initialization:}  $ \bm{\lambda}^0$, $\bm{\mu}^{0}$,  $\bm{\bar{X}}^0$,  $\bm{Z}^0$ and $\rho_0$,   and set $k \rightarrow 0$.
		\State \textbf{Repeat:}
		\State \textbf{{\em{Primal Update}}}: \\
		\quad {Update the consensus variables $\bm{Z}$, i.e., }
		\begin{equation} \label{(6)}
		\begin{split}
		\bm{Z}^{k+1}=\arg \min_{ \bm{Z} \in \mathcal{X} }  \mathbb{L}_{\rho_{k}}(\bm{Z}, \bm{\bar{X}}^k, \bm{\mu}^k).
		\end{split}
		\end{equation}
		\State \quad {Update the primal decision variables $\bm{\bar{X}}$  in parallel, i.e.,}
		\begin{equation} \label{(7)}
		\begin{split}
		\bm{\bar{X}}_i^{k+1} \!=\!\arg \min_{\bm{\bar{X}}_i} \mathbb{\tilde{L}}^i_{\rho_{k}} (\bm{\bar{X}}_i, &\bm{\bar{X}}^k_i, \bm{Z}^{k+1}, \bm{\lambda}^k_i, \bm{\mu}^{k}_i), \\
		&  \forall i \in \mathcal{N}.\\
		\end{split}
		\end{equation}	
		 \State \textbf{Dual update:}~{Update Lagrangian multipliers $\bm{\lambda}$, $\bm{\mu}$, i.e.,}
		\begin{equation}
		\begin{split}
		&\bm{\lambda}_i^{k+1}=\bm{\lambda}_i^k +\rho_{k}  h_i(\bm{X}^{k+1}_i),\\
		&\bm{\mu}^{ k+1}_i=\bm{\mu}^{k}_i+\rho_{k} \big( \bm{A}_i \bm{\bar{X}}_i^{k+1}\!-\!\bm{E}_i \bm{Z}^{k+1} \big), \\
		&\quad \quad \quad \quad \quad  \forall i \in \mathcal{N}.
		\end{split}
		\end{equation}
   \State \textbf{Adaptive step:}~{Update the penalty factor $\rho$}, i.e.,  
   if  $\bm{\bar{X}}^{k+1} \in \bm{\mathcal{X}}^{\eta}$, 
 set $\rho_{k+1}= \rho_{k}+\delta$, otherwise $\rho_{k+1}=\rho_{k}$. 
		\State If $R(k)\!\leq\!\epsilon$ stop, otherwise set $k\!=\!k\!+\!1$ and go to \textbf{Step} 2. 
	\end{algorithmic}
\end{algorithm}

\section{Performance and Convergence Analysis of PLDM}
This section discusses the performance and convergence of PLDM. First, we present the additional assumptions, basic definitions and propositions required. Then we illustrate the main results in \textbf{Theorem} 1 and \textbf{Theorem} 2. 


\subsection{Definitions and Lemmas}

\begin{lemma} \label{lemma1}
	(\textbf{Descent lemma}) (see \cite{bertsekas1997nonlinear}, \textbf{Proposition} A.24) Let $h:\mathbb{R}^d\!\rightarrow\!\mathbb{R}$ be  $M_{h}$-Lipschitz gradient continuous, we have
	\begin{displaymath}
	\begin{split}
	h(u) \leq h(v)\!+\!\langle u-v, \nabla h(u) \rangle\!+\!&\frac{M_h}{2} \Vert u-v\Vert^2, ~\forall u, v \in \mathbb{R}^d.
	\end{split}
	\end{displaymath}
	
\end{lemma}

\begin{lemma} \label{lemma2}
	(\textbf{Sufficient decrease property}) (see \cite{bolte2014proximal}, \textbf{Lemma} 2) Let $h:\mathbb{R}^d \!\rightarrow \!\mathbb{R}$ be $M_h$-Lipschitz gradient continuous and  $\sigma: \mathbb{R}^d \rightarrow \mathbb{R}$ be a proper and lower semicontinuous function with $\inf_{\mathbb{R}^d} \sigma>-\infty$.   If $prox_t^{\sigma}=\arg \min \{ \sigma(u)+\frac{t}{2} \Vert u-x\Vert^2 \}$ denotes the proximal map associated with $\sigma$, then for any  fixed $t>M_h$, $u \in $ dom $\sigma$, and  $u^+$ defined by
	\begin{equation}
	u^+ \in  {prox}^{\sigma}_t \Big ( u-\frac{1}{t} \nabla h(u) \Big)
	\end{equation}
	we have 
	$h(u^+)+\sigma(u^+) \leq h(u)+\sigma(u)-\frac{1}{2} (t-M_h) \big \Vert u^+-u \big \Vert^2$.
\end{lemma}

\begin{definition} \label{Def1}
	(\textbf{Normal cone}) (see \cite{hours2014augmented}) Let $\mathcal{\bm{X}}\!\!\subseteq\!\!\mathbb{R}^d$ be a convex set, the normal cone of $\mathcal{\bm{X}}$ is the set-valued mapping
	\begin{displaymath}
	\begin{split}
	\mathcal{N}_{\mathcal{\bm{X}}}(\bar{x}) =\left \{
	\begin{array}{lcl} \big\{  g \in \mathbb{R}^d | \forall x\in\mathcal{\bm{X}}, g^T(x-\bar{x}) \leq 0 \big\},~~ &\textrm{if}~\bar{x} \in \mathcal{\bm{X}}~~\\
	\varnothing, &\textrm{if}~\bar{x} \notin \mathcal{\bm{X}}.~~
	\end{array}  
	\right.
	\end{split}
	\end{displaymath}
	
\end{definition}

\begin{definition}  \label{Def2}
	(\textbf{Critical point}) (see \cite{boggs1995sequential, kanzow2018augmented, bot2018proximal}) Considering the following  problem (\textbf{P}):
	\begin{equation*}
	\begin{split}
	&(\textbf{\textrm{P}})~~\Big \{\min_{x \in \mathbb{R}^d}  f(x)~|~\bm{g}(x) \leq \bm{0},   \bm{h} (x)=\bm{0}. \Big\}
	\end{split}
	\end{equation*}
	where  the objective $f: \mathbb{R}^d \rightarrow \mathbb{R}$, and the constraints $\bm{g}=(g_1, g_2, \cdots, g_m) $ with $ g_i: \mathbb{R}^d \rightarrow \mathbb{R}$, $\bm{h}=(h_1, h_2, \cdots, h_\ell) $  with
	$h_i: \mathbb{R}^d \rightarrow \mathbb{R}$ are continuously differentiable. 
	The critical points (i.e., KKT points) of  problem (\textbf{P}) denote its feasible points satisfying first-order optimality condition described by  
	{\small{ 
	\begin{equation*} \label{critical points}
	\begin{split}
	&\textrm{crit} ( \bm{P})\!=\!\left\{
	\begin{array}{c}
	x \in \mathbb{R}^d \\
	\bm{\lambda} \in \mathbb{R}_{+}^m\\
	\bm{\mu} \in \mathbb{R}^{\ell}\\
	\end{array}
	\Bigg |
		\begin{array}{c}
	\nabla f(x)\!+\! (\nabla \bm{g}(x))^T\bm{\lambda} \!+\!(\nabla h(x))^T\bm{\mu} \!=\!\bm{0}. \\
	\bm{g}(x) \geq 0, ~	\bm{h}(x)=0. \\
	\bm{\lambda}_i g_i(x)=0, ~i=1, 2, \cdots, m.\\
	\end{array}
	\right\} \\
	 \end{split}
	 \end{equation*}	}}
	 As a direct extension, we can define the collection of  $\epsilon$-critical points   $\textrm{crit}_{\epsilon} ( \bm{P})$ by replacing $\bm{0}$ on the right-hand side of the equalities (inequalities) with $\epsilon$.

\end{definition}

\begin{remark}
	For convex problems, the critical points (KKT points) are exactly the global optima. However, without convexity, 
	a critical  point can be a global optima, a  local optima, or  a ``saddle point".  For problem \eqref{P5} discussed in this paper, the 
	critical points can be described by
	{\small{
	\begin{equation*}
	\begin{split}
	&\textrm{crit \big(~problem (5) \big)}= \\
	&\left\{   
	\begin{array}{c}
	\bm{\bar{X}} \in \mathbb{R}^{\sum_{i \in \mathcal{N}} \bar{n}_i} \\
	\bm{\lambda} \in \mathbb{R}^{\sum_{i \in \mathcal{N}} m_i} \\
   \bm{\mu} \in \mathbb{R}^{\sum_{i \in \mathcal{N}} 2 \bar{n}_i} \\
	\end{array}
	 \Bigg |
	\begin{array} {c}
		\nabla_{\bm{\bar{X}}} \mathbb{L}_{\rho} (\bm{\bar{X}}, \bm{Z}, \bm{\lambda}, \bm{\mu})=\bm{0}.  \\
		\nabla_{\bm{\bar{X}}} \mathbb{L}_{\rho} (\bm{\bar{X}}, \bm{Z}, \bm{\lambda}, \bm{\mu})+\mathcal{N}_{\mathcal{\bm{X}}}(\bm{Z})=\bm{0}.  \\
		\bm{h}_i(\bm{\bar{X}_i})=\bm{0}, ~\forall i \in \mathcal{N}. \\
		\bm{A}_i \bm{\bar{X}}_i-\bm{E}_i\bm{Z}=0, ~\forall i \in \mathcal{N}.
		\end{array}
		\right\}
		\end{split}
	\end{equation*}
}}
\end{remark}

\begin{definition} \label{Def3}
	(\textbf{Uniform Regularity}) (see \cite{bolte2018nonconvex}) Let $\mathcal{\bm{X}} \subseteq  \mathbb{R}^m$ and  $h:\mathbb{R}^m \rightarrow \mathbb{R}^n$ be  continuously differentiable, we claim $h$ as uniformly regular over $\mathcal{\bm{X}}$ with a positive constant $\theta$ if 
	\begin{equation}
	\begin{split}
	\big \Vert (\nabla h (x) )^T \bm{v}\big \Vert \geq \theta \Vert \bm{v} \Vert, \forall x \in \mathcal{\bm{X}}, \bm{v} \in \mathbb{R}^n
	\end{split}
	\end{equation}
\end{definition}

\begin{remark}
	For a given $x \in \mathcal{X}$,  asserting that $\nabla h(x)$ is uniformly regular with $\theta$ ($\theta>0$) is equivalent to $$  \gamma(F, x)=\min_{\Vert \bm{v} \Vert=1} \Big\{ \Vert (\nabla h(x))^T \bm{v}  \Vert  \Big\} > 0.$$
	
	Equivalently,  the mapping $(\nabla h(x))^T$ is supposed to be surjective or $\nabla h(x) (\nabla h(x))^T$ is positive definite. 
	one may note that $\nabla h(x) (\nabla h(x))^T$ is always positive semidefinite. Therefore, the uniform regularity of $\nabla h(x)$ requires that the minimum eigenvalue of $\nabla h(x) (\nabla h(x))^T$ is positive.
   Besides, we note that if $\nabla h(x)$ has full row rank,  $\nabla h(x)$ will be uniformly regular with an existing positive constant.
\end{remark}

The next definition is a property that many analytical  functions hold and plays a central role in the convergence analysis of  nonconvex optimization (see \cite{bolte2014proximal, li2015accelerated}). 

\begin{definition} \label{Def4}
	(\textbf{Kurdyka-Łojasiewicz} (KŁ) property) (see \cite{bolte2014proximal}) We say function $h: \mathbb{R}^p \rightarrow \mathbb{R} \cup \{+\infty\}$  has the KŁ property 
	at $x^* \in \textrm{dom} ~\partial h$, if there exist $\eta \in (0, +\infty)$, a neighborhood 
	$U$ of $x^*$, and a continuous concave function $\varphi: [0, \eta) \rightarrow \mathbb{R}_+$ such that:
	
	\noindent
	(i) $\varphi(0)=0$ and $\varphi$ is differentible on $(0, \varphi)$;
	
	\noindent
	(ii) $\forall s\in (0, \eta)$, $\varphi^{'} (s)>0$;
	
	\noindent
	(iii) $\forall s \in U \cap \big\{x: h(x^*) < h(x) < h(x^*)+\eta \big\}$. Then the following KŁ inequality holds:
	$$  \varphi^{'} \big( h(x)-h(x^{*})\big)\cdot\textrm{dist} \big( 0, \partial h(x)\big) \geq 1,$$ 
	and function $\varphi$  is called  a desingularizing function of  $h$ at $x^*$. 
	
\end{definition}

For \textbf{Algorithm} \ref{CADMM}, we need two additional assumptions to guarantee convergence. 
\begin{itemize}
	\item [{\em{(A6)}}]The  Lagrangian multipliers sequences  $\{ \bm{\lambda}_i^k\}_{k \in \mathbb{N}}$  ($i \in \mathcal{N}$) and $\{\bm{\bm{\mu}_i^k} \}_{k \in \mathbb{N}}$ ($i \in \mathcal{N}$)  generated by \textbf{Algorithm} \ref{CADMM} are bounded by $M_{\bm{\lambda}_i}$ and $M_{\bm{\mu}_i}$, respectively. 
	\item [{\em{(A7)}}] The original problem (\ref{main problem}) is well-defined. Thus for a proper penalty factor $\rho<+\infty$, the primal problem in \eqref{eq:pp} is lower bounded as
	\begin{equation}
	\begin{split}
	\inf_{\bm{\bar{X}}, \bm{Z} \in \bm{\mathcal{X}}} \mathbb{L}_{\rho} (\bm{\bar{X}}, \bm{Z}, \bm{\lambda}, \bm{\mu}) >-\infty.
	\end{split}
	\end{equation}
\end{itemize}


\subsection{PLDM Convergence}
In this subsection, we first illustrate \textbf{Propositions} 1-9 that required to prove the convergence of the PLDM.  After that  we present the main results  in \textbf{Theorem} 1.  

\begin{proposition} \label{prop2}
	Let $\bm{\mathcal{\bar{X}}}_i^{\eta}=P_{\mathbb{R}^{\bar{n}_i}}[\bm{\mathcal{\bar{X}}}^{\eta}]$ ($i \in \mathcal{N}$), 
	there exists a finite iteration $\underline{k}$ such that
	\begin{center}
		\begin{itemize}
			\item [{\em{(a)}}] $\{\bm{\bar{X}}^k_i \}_{k \in \mathbb{N}, k \geq  \underline{k}} \in \bm{\mathcal{\bar{X}}}^{\eta}_i~ (\forall i \in \mathcal{N}).$\\
			\item [{\em{(b)}}] $\rho_{k}=\rho_{\underline{k}}, ~~\forall k \geq \underline{k}.$
			
		\end{itemize}
	\end{center}
\end{proposition}

\noindent{\it{Proof:}} Refer to \textbf{Appendix} A. 

\begin{remark}
	\textbf{Proposition} 1(a) illustrates that for any pre-defined sub-feasible region for the non-linear constraints \eqref{eq:5a},  the local  decision variables ($\bm{\bar{X}}^k_i$) for the agents will be forced into the sub-feasible region ($\bm{\mathcal{\bar{X}}}_i^{\eta}$) within finite iterations  ($\underline{k}$).  Meanwhile, the penalty factor ($\rho$) will stop  increasing  as described  in \textbf{Proposition} 1(b).
\end{remark}
\vspace{0.1 in}
\begin{proposition} \label{prop1}
	 $g_i (\bm{\bar{X}}_i, \bm{\lambda}^k_i, \rho_{k})$ ($ \forall k \in \mathbb{N}$)  is  Lipschitz continuous w.r.t $\bm{\bar{X}}_i$  over $\bm{\mathcal{\bar{X}}}_i^{\eta}$ with constant  $L_{g_i}=\max\{ L_{f_i}+L_{\phi_i}+ M_{{\lambda}_i}  L_{h_i}+\rho_{\underline{k}} C_{h_i}, 2M_i\}$ ($ C_{h_i}$ denotes the upper bound of $ \Vert ( \nabla_{\bm{X}_i} \!h_i(\bm{X}_i))^T h_i(\bm{X}_i) \Vert$ over $\bm{\mathcal{\bar{X}}}^\eta_i$).
\end{proposition}

\noindent{\it{Proof:}} Refer to \textbf{Appendix} B. 
\vspace{0.1 cm}

\begin{remark}
	\textbf{Proposition} 2 presents the smoothness of the function $g_i(\bm{\bar{X}}_i, \bm{\lambda}^k_i, \rho_k)$ ($\forall k \in \mathbb{N}$) over the bounded sub-feasible region $\bm{\mathcal{\bar{X}}}_i^{\eta}$.  
\end{remark}
\vspace{0.1 in}







\begin{proposition} \label{prop4}
	Let $\{\bm{\bar{X}}^k\}_{k \in \mathbb{N}, k \geq \underline{k}}$ and $\{\bm{Z}^k\}_{k \in \mathbb{N}, k \geq \underline{k}}$ be the sequences generated by \textbf{Algorithm} \ref{CADMM}, then we have 
	\begin{equation} \label{(11)}
	\begin{split}
	& \mathbb{L}_{\rho_{\underline{k}}} (\bm{\bar{X}}^{k+1}, \bm{Z}^{k+1}, \bm{\lambda}^k, \bm{\mu}^k) \leq  \mathbb{L}_{\rho_{\underline{k}}} (\bm{\bar{X}}^{k}, \bm{Z}^{k}, \bm{\lambda}^k, \bm{\mu}^k)\\
	&\quad \quad -\sum_{i=1}^N \frac{1}{2} (c_i^k-L_{g_i}) \Vert \bm{\bar{X}}_i^{k+1}-\bm{\bar{X}}^k_i\Vert^2 
	\end{split}
	\end{equation}
\end{proposition}
\noindent{\it{Proof:}} Refer to  \textbf{Appendix} C. 
\vspace{0.1 cm}

\begin{remark}
	\textbf{Proposition} 3 provides the lower bound for the ``decrease" of the AL function w.r.t. the primal updates.  One may note that the AL function is non-increasing w.r.t the primal update at each iteration with a step-size $c_i^k \geq L_{g_i}$.
\end{remark}
\vspace{0.1 in}

\begin{proposition} \label{prop5}
	Let $\{\bm{\bar{X}}^k\}_{k \in \mathbb{N}, k \geq \underline{k}}$ and $\{\bm{Z}^k\}_{k \in \mathbb{N}, k \geq \underline{k}}$ be the sequances generated by \textbf{Algorithm} \ref{CADMM}, we have 
	\begin{equation} 
	\begin{split}
	\big \Vert \nabla_{\bm{\bar{X}}_i }  &\mathbb{L}_{\rho_{\underline{k}}} (\bm{\bar{X}}^{k+1},  \bm{Z}^{k+1}, \bm{\lambda}^k, \bm{\mu}^k) \big \Vert \\
	&\leq  \big(L_{g_i}+c_i^k \big) \Vert \bm{\bar{X}}^{k+1}_i-\bm{\bar{X}}^k_i\Vert, ~\forall i \in \mathcal{N}. \\
	\end{split}
	\end{equation}	
\end{proposition}
\noindent{\it{Proof:}} Refer to \textbf{Appendix} D. 
\vspace{0.1 cm}

\begin{remark}
\textbf{Proposition} 4 provides the upper bound for the subgradients of the AL functions after primal update at each iteration.
\end{remark}
\vspace{0.1 in}

\begin{proposition} \label{prop9}
	Let $\{\bm{\bar{X}}^k\}_{k \in \mathbb{N}, k\geq \underline{k}}$  be the sequence generated by  \textbf{Algorithm} \ref{CADMM}, we have
	\begin{eqnarray}
	\begin{split}
	\Vert \bm{\gamma}^{k+1}_i - \bm{\gamma}^{k}_i \Vert \leq \Omega^i_1 \Vert \bm{\bar{X}}^{k+1}_i-\bm{\bar{X}}^k_i\Vert+& \Omega^i_2 \Vert  \bm{\bar{X}}^k_i-\bm{\bar{X}}^{k-1}_i \Vert,\\
	  & ~\forall i \in \mathcal{N}.
	\end{split}
	\end{eqnarray}
	where $\bm{\gamma}_i=\big((\bm{\lambda}_i)^T, (\bm{\mu}_i)^T\big)^T$ denoting the augmented Lagrangian multipliers for Agent $i$. We have $\Omega^i_1= \big(L_{g_i} +c_i^k+L_{f_i}+L_{\phi_i}+M_{h_i} M_{\bm{\gamma}_i} \big)/\theta$  and 
	$\Omega^i_2=\big(L_{g_i}+c_i^{k\!-\!1})/\theta$ ($i \in \mathcal{N}$).
\end{proposition}

\noindent{\it{Proof:}} Refer to \textbf{Appendix} E. 
\vspace{0.1 cm}

\begin{remark}
	\textbf{Proposition} 5 provides the upper bound for the difference of the Lagrangian multipliers over two successive iterations.
\end{remark}
\vspace{0.1 in}









To illustrate the convergence of the proposed PLMD, we need to resort to a Lyapunov function defined as
		 \begin{equation*}
\begin{split}
\Phi_{\beta}(\bm{\bar{X}}, \bm{Z}, \bm{\lambda}, \bm{\mu}, \bm{U})&=\mathbb{L}_{\rho} (\bm{\bar{X}}, \bm{Z}, \bm{\lambda}, \bm{\mu})
+\beta \Vert  \bm{\bar{X}}-\bm{U}\Vert^2,\\
\end{split}
\end{equation*}
where $\beta>0$ and $\bm{U}\in \mathbb{R}^{\sum_{i \in \mathcal{N}} \bar{n}_i}$.  One may note that the Lyapunov function closely relates to the AL function except for the extra term $\beta \Vert  \bm{\bar{X}}-\bm{U}\Vert^2$.
For the Lyapunov function under the sequences generated by \textbf{Algorithm} \ref{CADMM}, we have the following proposition. 

\begin{proposition} \label{prop11}
	Let $\Big\{\bm{W}^k\!=\!\big(\bm{\bar{X}}^k,  \!\bm{Z}^k, \! \bm{\lambda}^k, \! \bm{\mu}^k, \bm{\bar{X}}^{k-1} \big) \Big\}_{k \in \mathbb{N}, k\geq \underline{k}}$ be the sequence generated by  \textbf{Algorithm} \ref{CADMM}, we have
	\begin{equation} \label{(27)}
	\begin{split}
	\Phi_{\beta_k}&(\bm{W}^k)-\Phi_{\beta_{k+1}}(\bm{W}^{k+1}) \\
	&  \geq b_1^k  \Vert \bm{\bar{X}}^{k+1}-\bm{\bar{X}}^{k} \Vert^2 +b^k_2   \Vert \bm{\bar{X}}^{k}-\bm{\bar{X}}^{k-1}\Vert^2,\\
	\end{split}
	\end{equation}
	where 
	$ \Phi_{\beta_k}( \bm{W}^k)\!=\!\mathbb{L}_{\rho_{\underline{k}}} (\bm{\bar{X}}^k, \bm{Z}^k, \bm{\lambda}^k, \bm{\mu}^k)
	+\beta_k \Vert  \bm{\bar{X}}^k-\bm{\bar{X}}^{k-1}\Vert^2$,
	with $\beta_k$ a positive parameter. Besides, we have
	\begin{equation*}
	\begin{split}
	&b^k_1=\min \limits_{i} \Big \{\frac{1}{2} (c_i^k-L_{g_i})- \frac{2 (\Omega^i_1)^2}{ \rho_{\underline{k}} }-\beta_{k+1} \Big\},\\
	& b^k_2=\min \limits_{i} \Big \{\beta_k-\frac{2 (\Omega^i_2)^2}{ \rho_{\underline{k}} }\Big\}.
	\end{split}
	\end{equation*}
	
\end{proposition}

\noindent{\it{Proof:}} Refer to \textbf{Appendix} F. 
\vspace{0.1 cm}

\begin{remark}
	From \textbf{Proposition} 1, we have $\{ \bm{\bar{X}}_k \}_{k \in \mathbb{N}, k \geq \underline{k}} \in \bm{\mathcal{\bar{X}}}^{\eta}_i$. 
	To illustrate the convergence of PLDM over $ k \geq \underline{k}$, we resort to the Lyaponov function $\Phi_{\beta_k}(\bm{W}^k)$.  \textbf{Proposition} 6 provides the lower bound for the decrease of the Lyapunov function over successive iterations after iteration $\underline{k}$.
\end{remark}
\vspace{0.1 in}


\begin{proposition} \label{prop12}
	Let $\{ \bm{\lambda}^k \}_{k \in \mathbb{N}} $  and  $\{\bm{\bar{X}}^k\}_{k \in \mathbb{N}}$  be the sequance generated by  \textbf{Algorithm} \ref{CADMM}, then we have
	\begin{equation}
	\begin{split}
	& \lim_{k \rightarrow  +\infty} \Big \Vert \bm{\bar{X}}^{k+1} -\bm{\bar{X}}^k \Big \Vert \rightarrow 0.\\
	& \lim_{k \rightarrow  +\infty} \Big \Vert \bm{\bar{X}}^{k+1}_i -\bm{\bar{X}}^k_i \Big \Vert \rightarrow 0,\\
	& \lim_{k \rightarrow  +\infty} \Big \Vert \bm{\lambda}_i^{k+1} -\bm{\lambda}_i^k \Big \Vert \rightarrow 0,\\
	& \lim_{k \rightarrow  +\infty} \Big \Vert \bm{\mu}_i^{k+1} -\bm{\mu}_i^k \Big \Vert \rightarrow 0,  ~\forall i \in \mathcal{N},\\
	\end{split}
	\end{equation}
	provided with \textbf{Condition} (a):\\
	\begin{equation*}
	\begin{split}
	&b^k_1=\min \limits_{i} \Big \{\frac{1}{2} (c_i^k-L_{g_i})- \frac{2 (\Omega^i_1)^2}{ \rho_{\underline{k}} }-\beta_{k+1} \Big\} >0,\\
	&b^k_2=\min \limits_{i} \Big \{\beta_k-\frac{2(\Omega^i_2)^2}{ \rho_{ \underline{k}} } \Big\}>0.
	\end{split}
	\end{equation*}
	is satisfied.
\end{proposition}
\vspace{0.1 in}

\noindent{\it{Proof:}} Refer to \textbf{Appendix} G. 
\vspace{0.1 cm}

\begin{remark}
	\textbf{Proposition} 7 illustrates the boundedness of primal and dual sequences generated by \textbf{Algorithm} 1 under \textbf{Condition (a)}, which 
	 can be satisfied by selecting $\beta_k$ that $\max_i \{ \frac{2(\Omega^i_2)^2}{ \rho_{ \underline{k}} } \}  < \beta_k< \min_i \{ \frac{1}{2} (c_i^{k-1}-L_{g_i})- \frac{2 (\Omega^i_1)^2}{ \rho_{\underline{k}} } \}$.
\end{remark}
\vspace{0.1 in}

\begin{proposition} \label{prop13}
	Let $\{ \bm{W}^k \}_{k \in \mathbb{N}, k \geq \underline{k}} $ be the sequance generated by  \textbf{Algorithm} \ref{CADMM}, we have
	\begin{equation}
	\begin{split}
	\Vert \nabla \Phi_{\beta_{k+1}}  (\bm{W}^{k+1})\Vert & \leq b^k_3 \sum_{i=1}^N \Vert \bm{\bar{X}}^{k+1}_i-\bm{\bar{X}}^k_i\Vert \\
	&+ b^k_4 \sum_{i=1}^N  \Vert \bm{\bar{X}}^{k}_i-\bm{\bar{X}}^{k-1}_i\Vert\\
	\end{split}
	\end{equation}
	where
	 we have
	$B=\sup_{k \geq \underline{k}} \Vert \bm{F}_i(\bm{\bar{X}}_i^{k+1})  \Vert$.  
	$b^k_3=\max \limits_{i} \big\{L_{g_i}+c^k_i +\Omega^i_1 B+4\beta_{k+1}+\rho_{\underline{k}}+\frac{\Omega^i_1}{\rho_{\underline{k}}}  \big\}$ and $b^k_4=\max \limits_{i} \big\{  \Omega^i_2B +\frac{\Omega^i_2}{\rho_{\underline{k}}} \big\}$.

\end{proposition}

\noindent{\it{Proof:}} Refer to \textbf{Appendix} H. 
\vspace{0.1 cm}

\begin{remark}
	\textbf{Proposition} 8 provides the upper bound for the subgradient of the Lyapunov function $\Phi_{\beta_{k}}$. 
\end{remark}
\vspace{0.1 in}




Based on \textbf{Proposition}1-9, we have the following results for the convergence of the PLDM.

\begin{theorem} \label{theorem1}
	(\textbf{Convergence}) The PLDM described in \textbf{Algorithm} \ref{CADMM} converges to the critical points of problem \eqref{P5} provided with a step-size $c_k^i$ satisfying \textbf{Condition} (a) in \textbf{Proposition} 7.
\end{theorem}
\begin{proof}
	
	Based on \textbf{\emph{Proposition}} \ref{prop12} and \textbf{\emph{Proposition}} \ref{prop13}, we have
	\begin{equation} \label{pp:21}
	 \lim_{k \rightarrow +\infty} \big \Vert  \nabla \Phi_{\beta_{k+1}} (\bm{W}^{k+1})  \big \Vert \rightarrow 0.
	\end{equation}
	 
	 Specifically, we have 
	 \begin{equation} \label{pp:22}
	 	\begin{split}
	 	\nabla_{\bm{\bar{X}}} \Phi_{\beta_{k+1}}  (\bm{W}^{k+1})=&\!\nabla_{\bm{\bar{X}}} \mathbb{L}_{\rho_{\underline{k}}} (\bm{\bar{X}}^{k+1}, \bm{Z}^{k+1}, \bm{\lambda}^{k+1}, \bm{\mu}^{k+1})\\
	 	&+2 \beta_{k+1} (\bm{\bar{X}}^{k+1}-\bm{\bar{X}}^k)\\
	 \nabla_{\bm{Z}} \Phi_{\beta_{k+1}}  (\bm{W}^{k+1})=&\!\nabla_{\bm{Z}} \mathbb{L}_{\rho_{\underline{k}}} (\bm{\bar{X}}^{k+1},  \bm{Z}^{k+1}, \bm{\lambda}^{k+1}, \bm{\mu}^{k+1})\\
	  &\quad  +\mathcal{N}_{\mathcal{\bm{X}}}(\bm{Z}^{k+1})\\
	 	\nabla_{\bm{\gamma}} \Phi_{\beta_{k+1}}  (\bm{W}^{k+1})=&\!\nabla_{\bm{\gamma}} \mathbb{L}_{\rho_{\underline{k}}} (\bm{\bar{X}}^{k+1}, \bm{Z}^{k+1}, \bm{\lambda}^{k+1}, \bm{\mu}^{k+1})\\
	 		 	\nabla_{\bm{U}} \Phi_{\beta_{k+1}}  (\bm{W}^{k+1})=&2 \beta_{k+1} (\bm{\bar{X}}^{k+1}-\bm{\bar{X}}^k)\\
	 	\end{split}
	 \end{equation}
	 
	 By combining \eqref{pp:21} with \eqref{pp:22}, we have 
	 \begin{equation*}
	 \begin{split}
	 	& \lim_{k \rightarrow +\infty} \nabla_{\bm{\bar{X}}}~\mathbb{L}_{\rho_{\underline{k}}} (\bm{\bar{X}}^{k+1}, \bm{Z}^{k+1}. \bm{\lambda}^{k+1}, \bm{\mu}^{k+1})=0. \\
		&\lim_{k \rightarrow +\infty} \nabla_{\bm{Z}}~\mathbb{L}_{\rho_{\underline{k}}} (\bm{\bar{X}}^{k+1}, \bm{Z}^{k+1}, \bm{\lambda}^{k+1}, \bm{\mu}^{k+1})+\mathcal{N}_{\mathcal{\bm{X}}}(\bm{Z}^{k+1})=0. \\
		&\lim_{k \rightarrow +\infty}  h_i(\bm{\bar{X}_i}^{k+1})=0, ~\forall i \in \mathcal{N}. \\
		&\lim_{k \rightarrow +\infty} \big(\bm{A}_i \bm{\bar{X}}^{k+1}_i-\bm{E}_i \bm{Z}^{k+1} \big)=0, ~\forall i \in \mathcal{N}. \\
		\end{split}
	 \end{equation*}
	 	
	 	According to \textbf{Definition} \ref{Def2}, the  above implies that PLDM will converge to the critical points of problem (\ref{P4}).
	  
\end{proof}

\begin{theorem} \label{theorem2}
	 The critical points  of problem \eqref{P5} obtained from \textbf{Algorithm} \ref{CADMM}  are the $\epsilon$-critical points of problem \eqref{P4}. 
\end{theorem}
\begin{proof}
	We first denote the critical points of problem \eqref{P5} as $\bm{W}^*=\big( \bm{\bar{X}}^*,  \!\bm{Z}^*, \! \bm{\lambda}^*, \! \bm{\mu}^* , \bm{\bar{X}}^*\big)$.
	Based on \textbf{Theorem} \ref{theorem1}, we have $\bm{W}^k \rightarrow \bm{W}^*$ with $k \rightarrow +\infty$.
		Based on the definition of $\epsilon$-critical points (see \textbf{Definition} 2),  we only need to prove  $\Vert \bm{Y}_i^{*} \Vert \leq \epsilon$ to illustrate the results.
	 
	 This can be illustrated in twofolds.
	 
	  \emph{i)} according to \emph{\textbf{proposition}} 6, we have  $\Phi_{\beta_k} (\bm{W}^{k})$ is non-increasing  w.r.t the iteration $k$. Therefore, we have $\Phi_{\bar{\beta}} (\bm{W}^*)  \leq \Phi_{\beta_0} (\bm{W}^0)$ with $k \in \mathbb{N}$ (bounded), where we denote $\beta_k \rightarrow \bar{\beta}$ with $k \rightarrow +\infty$.
	 
	 \emph{ii)} Based on \textbf{Theorem} \ref{theorem1}, we have $\Phi_{\bar{\beta}} (\bm{W}^*)=\sum_{i=1}^N \tilde{f}_i(\bm{X}^*_i) +\sum_{i=1}^N \phi_i (\bm{X}^*_i )+\sum_{i=1}^N M_i \Vert \bm{Y}^*_i\Vert^2 $ as $h_i(\bm{\bar{X}_i}^{*})=0$ and $\big(\bm{A}_i \bm{\bar{X}}^{*}_i-\bm{E}_i \bm{Z}^{*} \big)=\bm{0}$ ($\forall i \in \mathcal{N}$).   
	 
	 Thus by combining \emph{i)} and \emph{ii)}, we imply $$\Vert \bm{Y}^*_i \Vert \leq \sqrt{ \frac{ \Phi_{\beta_0} (\bm{W}^0) -\sum_{i=1}^N\tilde{f} (\bm{X}^*_i)- \sum_{i=1}^N  \phi_i (\bm{X}^*_i) }{ \min\limits_{i} M_i \cdot N} } \leq \epsilon,$$  $~\forall i \in \mathcal{N}$ with $M_i$ sufficiently large.
	 

\end{proof}

\subsection{Convergence Rate}
With the  global convergence of  the PLDM studied, this section discusses the local convergence rate of the method. Before we present the main results, we illustrate the following propositions that  referred  to.

\begin{proposition} \label{prop16}
Let $\{ \bm{W}^k \}_{k \in \mathbb{N}, k \geq \underline{k}} $ be the sequance generated by  \textbf{Algorithm} \ref{CADMM}, then we have
\begin{equation} 
\begin{split}
\Big [\textrm{dist} &\big( \nabla \Phi_{\beta_{k+1}}( \bm{W}^{k+1}), 0\big) \Big ]^2\\
& \!\leq \!{2N}{ \nu_k } \big( \Phi_{\beta_k}( \bm{W}^k)-\Phi_{\beta_{k+1}}( \bm{W}^{k+1}) \big)
\end{split}
\end{equation}
provided with the step-size $c^k_i$ and the parameter $\beta_{k}$ satisfying  \textbf{Condition} (b):\\
$$  \nu_k=\frac{(b^k_4)^2}{b^k_2} \leq \frac{(b^k_3)^2}{b^k_1} $$
where $\nu_k$ is the positive parameter that need to be decided for \textbf{Algorithm} \ref{CADMM}.
\end{proposition}

\begin{remark}
	\textbf{Proposition} 9 gaps the subgradients of  the Lyapunov function $\Phi_{\beta_{k+1}}( \bm{W}^{k+1})$  in terms of  its value decrease over successive iterations. To guarantee \textbf{Condition} (b), the step-size $c^k_i$ and  the parameter $\beta_k$  can be selected following the procedures below:
	
	Based on \textbf{Proposition} \ref{prop11} and \textbf{Proposition} \ref{prop13}, we have
	\begin{displaymath}
	\begin{split}
	&b^k_1=\min \limits_{i} \Big \{\frac{1}{2} (c_i^k-L_{g_i})- \frac{2 (\Omega^i_1)^2}{ \rho_{\underline{k}} }-\beta_{k+1} \Big\} \\
	&\quad \quad=\frac{1}{2} c^k_i-\max \limits_i \Big\{ \frac{1}{2} L_{g_i})+\frac{2 (\Omega^i_1)^2}{ \rho_{\underline{k}} }+\beta_{k+1}  \Big\}  \\
	&\quad \quad=\frac{1}{2} c^k_i-M \\
		\end{split}
	\end{displaymath}
	\begin{displaymath}
	\begin{split}
	&b^k_2=\min \limits_{i} \Big \{\beta_k-\frac{2 (\Omega^i_2)^2}{ \rho_{\underline{k}} } \Big\}\\
	&b^k_3=\max \limits_{i} \Big\{L_{g_i} \!+\!c^k_i\!+\!\Omega^i_1 B+4\beta_{k+1}\!+\!\rho_{\underline{k}}\!+\!\frac{\Omega^i_1}{\rho_{\underline{k}}} \Big\}\\
	&\quad \quad=c_i^k+\max \limits_{i} \Big\{ L_{g_i}\!+\!\Omega^i_1 B+4\beta_{k+1}\!+\!\rho_{\underline{k}}\!+\!\frac{\Omega^i_1}{\rho_{\underline{k}}} \Big\}  \\
	&\quad \quad=c_i^k+L\\
	&b^k_4=\max \limits_{i} \big\{  \big(B +\frac{1}{\rho_{\underline{k}}}\big) \Omega^i_2 \big\}
	\end{split}
	\end{displaymath}
	where we have $M=\max \limits_i \Big\{ \frac{1}{2} L_{g_i}+\frac{2 (\Omega^i_1)^2}{ \rho_{\underline{k}} }+\beta_{k+1}  \Big\}$ and $L=\max \limits_{i} \Big\{ L_{g_i}\!+\!\Omega^i_1 B+4\beta_{k+1}\!+\!\rho_{\underline{k}}\!+\!\frac{\Omega^i_1}{\rho_{\underline{k}}}   \Big\}$.
\end{remark}


We first  select a large enough positive parameter $\nu_k>0$.  By letting  $\nu_k=\frac{(b^k_4)^2}{b^k_2}=\frac{\max \limits_{i} \big\{  \big(B +\frac{1}{\rho_{\underline{k}}}\big)^2 (\Omega^i_2)^2 \big\}}{\min \limits_{i} \Big \{\beta_k-\frac{2 (\Omega^i_2)^2}{ \rho_{\underline{k}} } \Big\}}$, we can set
\begin{equation*}
\begin{split}
\beta_k&=2\Big(  \frac{(B+\frac{1}{\rho_{\underline{k}}})^2}{\nu_k}+\frac{1}{\rho_{\underline{k}}} \Big) \max_i \big\{ (\Omega^i_2)^2\big\}\\
&=2\frac{(1+B\rho_{\underline{k}})^2+\nu_k}{\nu_k \rho_{\underline{k}}} \max_i \big\{ (\Omega^i_2)^2\big\}.
\end{split}
\end{equation*}
Besides, if we select $\nu_k=\nu_{\underline{k}}$, we have 
$\beta_k=\beta_{\underline{k}}~~\forall k \geq \underline{k}.$

 Afterwards, by letting $\nu_k \geq (b^k_3)^2/b^k_1$, i.e.,   
 \begin{equation}
\begin{split}
\nu_k \geq \frac{(c^k_i+L)^2}{\frac{1}{2} c^k_i-M},   
\end{split}
 \end{equation}
we can select the step-size $c^k_i$ at each iteration $k$ as
 \small $c^k_i \in \Big[  \frac{(\nu_k-4M)-\sqrt{(\nu_k)^2-(8M+16L) \nu_k}}{4},  $ $ \frac{(\nu_k-4M)+\sqrt{(\nu_k)^2-(8M+16L) \nu_k}}{4} \Big] $ $ \cap (0, +\infty)$. \normalsize Besides, to ensure that at least one positive step-size exists,  the parameter $\nu_k$ should be selected as $\nu_k \geq (8M+16L) $ ($\forall k \geq \underline{k}$).





\begin{theorem} \label{theorem2}
	(\textbf{Convergence rate}) Suppose  the Lynapunov function $ \Phi_{\beta_k} (\bm{W}^{k})$ satisfy the KL property\footnote{Note that KL property is general for most analytical functions
	} with  the desingularising function of  the form   $\varphi(t)=\frac{C}{\theta} t^{\theta}$ for some $C>0$, $\theta\in (0, 1]$  (\cite{li2015accelerated}) and  $c^k_i$, $\beta_k$  are selected satisfying  \textbf{Condition} (b) in \textbf{Proposition} 9,
we have:
	\begin{itemize}
		\item [{\em{(i)}}] If $\theta=1$, \textbf{Algorithm} \ref{CADMM} terminates in finite iterations. 
		\item [{\em{(ii)}}] If $\theta \in (0, 1/2)$,  $\exists~ k_0, k_1, k_2 \in \mathbb{R}$, such that
		
		$\forall \max \{k_0, k_1\} \leq k \leq k_2$, we have $$\Delta^k \leq  \Big(\frac{2N C^2\nu_{k-1}}{1+2N C^2\nu_{k-1}} \Big)^{\frac{1}{1+2\theta}} \Delta^{k-1}.$$

		$\forall k \geq \max \{k_0, k_1, k_2\}$, we have 
		$$ \Delta^k  \leq \big ( \frac{2N C^2\nu_{k-1}}{1+2N C^2\nu_{k-1}} \big)^{\frac{1}{2-2\theta}}  \big(\Delta^{k-1}\big)^{\frac{1}{2-2\theta}}. $$
		
		\item [{\em{(iii)}}] If $\theta \in [1/2, 1)$,   $ \exists~k_3 \in \mathbb{N}$, 
		$\forall k \geq k_3$, such that  $$ \Delta^k \leq \frac{2N C^2\nu_{k-1}}{1+2N C^2\nu_{k-1}} \cdot \Delta^{k-1}$$
	\end{itemize}
where $\Delta^k =\Phi_{\beta_k}( \bm{W}^k)-\Phi^*$ denotes the ``suboptimality" (critial points or KKT points)  at iteration $k$, with $\Phi^*$ as the ``optimal" value of the Lynapunov function.

	
\end{theorem}

\begin{proof}

We define the set of critical point for the Lyapunov function $\Phi_{\beta_k}( \bm{W}^k)$ 
as $\bm{\omega}(\bm{W}^0)$ with $\bm{W}^0=(\bm{X}^0, \bm{X}^0, \bm{Z}^0, \bm{\lambda}^0, \bm{\mu}^0)$ as any given start point.  

According to \textbf{Theorem} \ref{theorem1}, we conclude that the sequence $\{\bm{W}^k\}_{k \in \mathbb{N}}$ generated by \textbf{Algorithm} \ref{CADMM} from any given  initial  point $\bm{W}^0$ will converge to the set of critical points denoted by  $\omega(\bm{W}^0) \neq  \varnothing $.

Based on \textbf{Proposition} \ref{prop12}, $\{\bm{W}^k\}_{k \in \mathbb{N}}$ is bounded (convergent), 
and for any given initial point $\bm{W}^0$,  $k\in+\infty$, there exist a limit point $\overline{\bm{W}}=(\bm{\bar{X}},  \bm{\bar{Z}}, \bm{\bar{\lambda}}, \bm{\bar{\mu}}, \bm{\bar{X}})$.

Since the Lyapunov function is continuous, we have
\begin{equation}
\begin{split}
\lim_{k \rightarrow +\infty} \Phi_{\beta_k} (\bm{W}^k)=\Phi_{\bar{\beta}} (\bm{\overline{W}})
\end{split}
\end{equation}

As illustrated, $\Phi_{\beta_k}(\bm{W}^k)$  is lower bounded, i.e., $\lim\limits_{k \rightarrow +\infty }\Phi_{\beta_k}(\bm{W}^k)\!\!>\!\!\!-\infty$  and non-increasing w.r.t the iteration $k$  (\textbf{Proposition} \ref{prop11}), thus we have 
\begin{equation}
\begin{split}
\lim_{k \rightarrow +\infty} \Phi_{\beta_k} (\bm{W}^k)=\Phi_{\bar{\beta}} (\bm{\bar{W}})=\underline{\Phi}
\end{split}
\end{equation}
where $\underline{\Phi}$ denotes the limit  of $\Phi_{\beta_k} (\bm{W}^k)$ from the initial point $\bm{W}^0$. 

The above implies that  $\Phi_{\beta_k}( \bm{W}^k)$  is constant on $\omega(\bm{W}^0)$.  

From \textbf{\textrm{Theorem}} \ref{theorem1}, we have $$\lim_{k \rightarrow +\infty} \textrm{dist}  \big (\nabla  \Phi_{\beta_k}( \bm{W}^k), 0\big)=0$$

That means $\forall \varepsilon$, $\exists k_0 \in \mathbb{N}$ that $\forall k \geq k_0$,  $$\textrm{dist} \Big(\nabla  \Phi_{\beta_k}( \bm{W}^{k}), 0\Big) \leq \epsilon$$
 
For notation,  we denote $\Phi^*$ as the limit of $\Phi_{\beta_k} (\bm{W}^k)$ on $\omega(\bm{W}^0)$.
$ \Phi_{\beta_k}( \bm{W}^k)$ is non-increasing w.r.t. the iteration $k$,  thus $\forall \eta \geq 0$, there exist $k_1 \in \mathbb{N}$ such that $\forall k \geq k_1$, we have   $$\Phi_{\beta_k}( \bm{W}^k) \leq \Phi^*+\eta$$


Further, $\forall k\geq \max\big\{k_0, k_1\big\}$, we have 
\begin{equation}
\begin{split}
\bm{W}^k \in &\Big\{ \bm{W}: \textrm{dist} \big( \nabla \Phi_{\beta_k}( \bm{W}^k), 0\big)\leq \epsilon \Big \} \\
&\quad \cap \Big\{ \bm{W}: \Phi^* <\Phi_{\beta_k}( \bm{W}^k ) < \Phi^*+\eta \Big\}=\Omega_{\varepsilon, \eta}
\end{split}
\end{equation}

Since the Lyapunov function $\Phi_{\beta_k}( \bm{W}^k)$ posess the KL property,  based on \textbf{Definition} 4, we have
\begin{equation} \label{(55)}
\begin{split}
1\leq \Big[ \varphi^{'} \big( \Phi_{\beta_k}( \bm{W}^k)-\Phi^*\big) \Big]^2 \Big [\textrm{dist} \big( 0, \nabla \Phi_{\beta_k}( \bm{W}^k) \big) \Big]^2 
\end{split}
\end{equation}

By combining (\ref{(55)}) with \textbf{Proposition} \ref{prop16},  we have 
\small
\begin{equation} \label{29}
\begin{split}
1\!\leq\!\Big[ \varphi^{'}\! \big( \Phi_{\beta_k}( \bm{W}^k)\!-\!\Phi^*\big) \Big]^2 \! {2N}{\nu_{k\!-\!1}}\Big[ \Phi_{\beta_k}( \bm{W}^{k\!-\!1})\!-\!\Phi_{\beta_{k\!+\!1}}( \bm{W}^{k}) \Big]
\end{split}
\end{equation}
\normalsize

Rearrange \eqref{29} based on  $\Delta^k =\Phi_{\beta_k}( \bm{W}^k)-\Phi^*$, we have 
\begin{equation} \label{(76)}
\begin{split}
1 &\leq [\varphi^{'}(\Delta^k)]^2 \cdot {2N}{\nu_{k-1}}\big[ \Delta^{k-1}-\Delta^{k} \big]
\end{split}
\end{equation}

Regarding the desingularising function,  we have $\varphi^{'}(t)=C t^{\theta-1}$. 
Therefore, we can derive from (\ref{(76)}) that 

$$1 \leq {2N C^2}{\nu_{k-1}} (\Delta^k )^{2(\theta-1)} \cdot \big[ \Delta^{k-1}-\Delta^{k} \big].$$

\noindent
\emph{(i)} If $\theta=1$,  we have
\begin{equation}
\begin{split}
1 \leq {2N C^2}{\nu_{k-1}} \cdot \big[ \Delta^{k-1}-\Delta^{k} \big]  \leq 0
\end{split}
\end{equation}

This above inequality is a contradiction and implies that the set $\Omega_{\varepsilon, \eta}=\varnothing$.
In other word, $\forall k \geq \max\{k_0, k_1\}$, we have $\Delta^k=0$, i.e., \textbf{Algorithm} \ref{CADMM} terminates in finite iterations. 
~\\
\noindent
\emph{(ii)} $\theta \in (0, 1/2)$, we have $1<2-2\theta<2$. 
\begin{equation}
\begin{split}
[\Delta^k]^2 &\leq {2N C^2}{\nu_{k-1}} \cdot [\Delta^{k-1}-\Delta^k] \cdot (\Delta^k)^{2\theta} \\
& \leq {2N C^2}{\nu_{k-1}} \Delta^{k-1} (\Delta^k)^{2\theta}-{2N C^2}{\nu_{k-1}} (\Delta^k)^{1+2\theta}\\
\end{split}
\end{equation}

Thus, we have
\begin{equation} \label{(71)}
\begin{split}
[\Delta^k]^2&+{2N C^2}{\nu_{k-1}}(\Delta^k)^{1+2\theta} \\
&\leq {2N C^2}{\nu_{k-1}} \Delta^{k-1} \cdot (\Delta^k)^{2\theta}\\
\end{split}
\end{equation}

Since $\Delta^k$ is non-increasing (\textbf{Proposition} \ref{prop11}),  there exist $k_3 \in \mathbb{R}$ that $\forall  k \leq k_2$, we have $\Delta^{k-1} \geq  \Delta^k\geq 1$.

In this case, it's easy to figure out that 
$$1<1+2\theta <2, ~~ 0<2\theta<1$$
$$ (\Delta^k)^2 \geq (\Delta^k)^{1+2\theta} $$ 
$$ (\Delta^k)^{2\theta} \leq (\Delta^{k-1})^{2\theta}$$

Therefore, we can derive from (\ref{(71)}) that 
\begin{equation*} 
\begin{split}
&(1+{2N C^2}{\nu_{k-1}}) (\Delta^k) ^{1+2\theta}  \leq \\
&\ [\Delta^k]^2+{2N C^2}{\nu_{k-1}}(\Delta^k)^{1+2\theta}  \leq {2N C^2}{\nu_{k-1}} \Delta^{k-1} (\Delta^k)^{2\theta} \\
&\quad \quad \quad  \leq {2N C^2}{\nu_{k-1}}(\Delta^{k-1})^{1+2\theta}\\
\end{split}
\end{equation*}
In this case,  $\forall \max \{k_0, k_1\} \leq k \leq k_2$, we have 
\begin{equation}
\begin{split}
\Delta^k \leq  \Big(\frac{2N C^2 \nu_{k-1}}{1+2N C^2\nu_{k-1}} \Big)^{\frac{1}{1+2\theta}} \Delta^{k-1}
\end{split}
\end{equation}
This implies \textbf{Algorithm} \ref{CADMM} shows linear convergence rate. 

As mentioned, $\forall k \geq k_2$, we have $ 0 \leq \Delta_k <1$. In this case, it's straight forward that 
$$1<1+2\theta <2$$
$$ (\Delta^k)^2 \leq  (\Delta^k)^{1+2\theta} $$ 

Similarly, we can derive from (\ref{(71)}) that 
\begin{equation} 
\begin{split}
&(1+{2N C^2}{\nu_{k-1}}) (\Delta^k) ^{2}   \leq [\Delta^k]^2+{2N C^2}{\nu_{k-1}}(\Delta^k)^{1+2\theta}\\
&\quad  \quad \quad  \leq {2N C^2}{\nu_{k-1}} \Delta^{k-1} (\Delta^k)^{2\theta}\\
\end{split}
\end{equation}
This implies that 
\begin{equation} 
\begin{split}
&(1+{2N C^2}{\nu_{k-1}}) (\Delta^k) ^{2-2\theta} \leq {2N C^2}{\nu_{k-1}} \Delta^{k-1} \\
\end{split}
\end{equation}
Equivalently, $\forall k \geq \max \{k_0, k_1, k_2\}$, we have 
\begin{equation} 
\begin{split}
& \Delta^k  \leq \Big ( \frac{2N C^2 \nu_{k-1}}{1++2N C^2\nu_{k-1}} \Big)^{\frac{1}{2-2\theta}}  \big(\Delta^{k-1}\big)^{\frac{1}{2-2\theta}} \\
\end{split}
\end{equation}
~\\
\noindent
\emph{(iii)} $\theta \in [1/2, 1)$, we have $0\!<\!2-2\theta\!\leq\!1$. As $\Delta^k$ is non-increasing,  there exist  $  k_3 \in \mathbb{R}$ that 
 $\forall k \geq k_3$, we have $$(\Delta^k)^{2(1-\theta)} \geq \Delta^k$$

Therefore, $\forall k \geq \max \{ k_0, k_1, k_3\}$, we have 
$$  \Delta^k \leq {2N C^2}{\nu_{k-1}} \cdot \big(\Delta^{k-1}-\Delta^{k} \big) $$
i.e.,
$$  \Delta^k \leq \frac{2N C^2\nu_{k-1}}{1+2N C^2\nu_{k-1}} \cdot \Delta^{k-1} $$
In this case,  \textbf{Algorithm} \ref{CADMM} presents linear convergence rate.

\end{proof}

According to \textbf{Theorem} \ref{theorem1}, the step-size should be selected to guarantee \textbf{Condition} (a). Nevertheless this is a challenging task to achieve efficiently in practice. 
To reduce computation and facilitate implementation, we can select the step-size by the backward linesearch procedures in  \textbf{Algorithm} \ref{stepsize searching}.  $\alpha$ denotes a small positive value that can be selected adaptively according to the problems. 

\begin{algorithm}
	\caption{Backward Linesearch for Stepsize } \label{stepsize searching}
	\begin{algorithmic}[1]
		\State  \textbf{Initialization:}  $c_i^0$.
		\State \textbf{Repeat:}
		\State If 
		\small
\begin{eqnarray}\label{(68)}
\begin{split}
&g_i(\bm{\bar{X}}_i^{k+1}, \bm{\lambda}^k, \rho_{k})\!+\!\alpha \Vert \bm{\bar{X}}^{k+1}_i  -\bm{\bar{X}}_i^k \Vert^2\!>\! g_i(\bm{\bar{X}}^k_i, \bm{\lambda}^k, \rho_{k}) \\
&\!+\!\langle  \nabla_{\bm{\bar{X}}_i} g_i(\bm{\bar{X}}^k_i, \!\bm{\lambda}^k, \! \rho_{k} ),  \bm{\bar{X}}^{k+1}_i\!-\!\bm{\bar{X}}_i^k  \rangle\!+\!\frac{c_i^k}{2} \big \Vert  \bm{\bar{X}}^{k+1}_i\!-\!\bm{\bar{X}}_i^k \big \Vert^2\\
\end{split}
\end{eqnarray}
\normalsize
		   set $c^{k+1}_i=c^i_k/ \delta$ and go to \textbf{Step} 2, otherwise stop.
		\State  \textbf{Output}  $c^k_i$.
	\end{algorithmic}
\end{algorithm}

\begin{corollary}  \label{remark1}
	Suppose 
	the step-size $c^k_i$ at each iteration $k$ is selected according to \textbf{Algorithm} \ref{stepsize searching}, the PLDM converges to the critical points of Problem \eqref{P5}. 
\end{corollary}
\begin{proof}
	Refer to \textbf{Appendix} C.
\end{proof}

\section{Numeric Experiments}
This section illustrates the PLDM's performance on {\em{(i)}} a simple numerical example with 2 agents and {\em{(ii)}} the application to multi-zone HVAC control. The numerical example helps illustrate the theoretical analysis and the application demonstrates its practical capabilities.

\subsection{A Numerical Example}
We consider a nonlinear and non-convex example with \textrm{2} agents given by
\begin{equation} \label{(70)}
\begin{split}
\min_{x_1, x_2}~~ x_1&+x_2+\frac{1}{2}x_1 \cdot x_2^2+\frac{1}{2} x_1 \cdot x_2^2\\
 s. t.&~~ x_1 \cdot x_2=2. \\
&~~~x_1^2+x_2^2=5. \\
&~~~0 \leq x_1 \leq 4.\\
&~~~0 \leq x_2 \leq 5. 
\end{split}
\end{equation}

The problem is well-posed with the (local) optima  $(2, 1)$. The problem is solved   by using PLDM  in a decentralized manner and our analysis is presented here. 
\small
\begin{displaymath}
\begin{split}
	\bm{F}_i(\bm{\bar{X}}_i)&=
\left(  
\begin{array}{ccc}
\nabla_{\bm{X}_i} h_i(\bm{X}_i) &\!\! & \bm{O}_{n_i} \\
&\bm{A}_i&\\
\end{array} 
\right)	\\
&=\left(  
\begin{array}{cccc}
\bm{X}_i(2)    & \bm{X}_i(1)    & 0    & 0 \\
2\bm{X}_i(1)  & 2\bm{X}_i(2)   & 0    & 0\\
1       &      0     &  1  & 0\\
0       &      1     &  0  & 1\\
\end{array} 
\right), ~i=\{1, 2\}.
\end{split}
\end{displaymath}
\normalsize
Note that $\bm{F}_i(\bm{\bar{X}}_i)$ has row full rank and regular  with $\bm{X}_i \neq \bm{0}$.  This implies  the satisfaction of  assumption {\emph{(A4)}}, which closely related to the convergence. The trajectory of the local decision variable sequence $\{\bm{X}_1\}_{k \in \mathbb{R}}$ is shown in Fig.  \ref{figure1}.
Fig. \ref{figure3} shows the subgradients of AL function, i.e.,  $\Vert \nabla \mathbb{L}_{\rho_{k}} (\bm{\bar{X}}^k, \bm{Z}^k, \bm{\lambda}^k, \bm{\mu}^k) \Vert $ and the Lyapunov function, i.e.,  $\Vert \nabla \Phi_{\beta_k} (\bm{W}^k) \Vert$ converging to zero w.r.t the iteration $k$. This additionally demonstrates the PLDM's convergence property illustrated  in \textbf{Theorem} 1. 

In addition, we study the PLDM's convergence rate by closely inspecting the Lyapunov function   $ \Phi_{\beta_k} (\bm{W}^k)$  w.r.t  the iteration $k$. As shown in Fig. \ref{figure4},  we see that the Lyapunov function $ \Phi_{\beta_k} (\bm{W}^k)$ approximately decrease w.r.t the iteration $k$ until the optima $ \Phi^*$ is reached. 
\begin{figure}
	\centering
	\includegraphics[width=3.0 in]{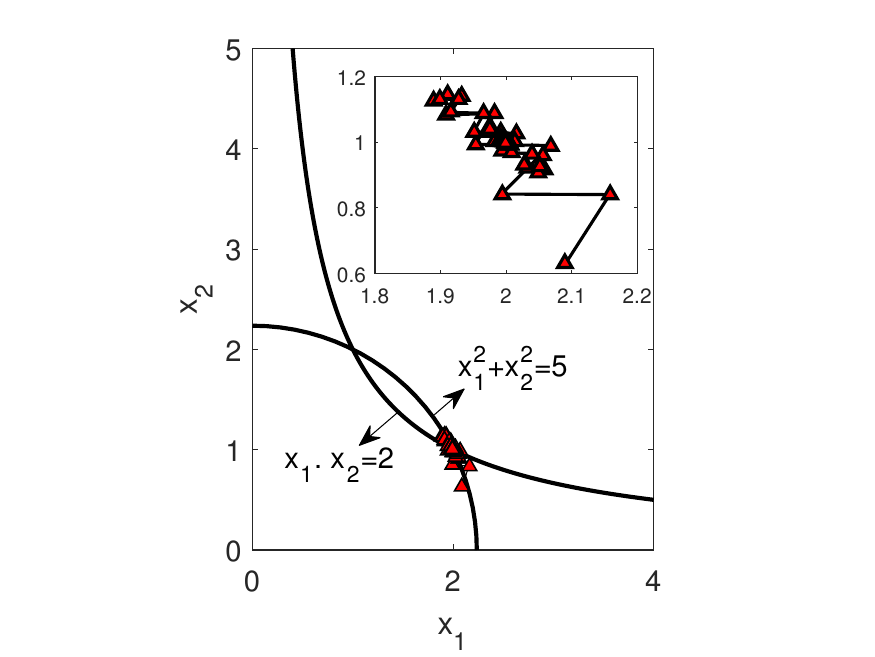}\\
 \caption{ the trajectory of the sequence $\{\bm{X}_1\}_{k \in \mathbb{N}}$ (the feasible points include $(1, 2)$ and $(2, 1)$ with the (local) optima $(2, 1)$).} \label{figure1} 
\end{figure}

\begin{figure}[htbp]
	\centering
	\subfigure[]{
		\begin{minipage}[t]{0.5\linewidth}
			\centering
			\includegraphics[width=1.8 in, height=1 in]{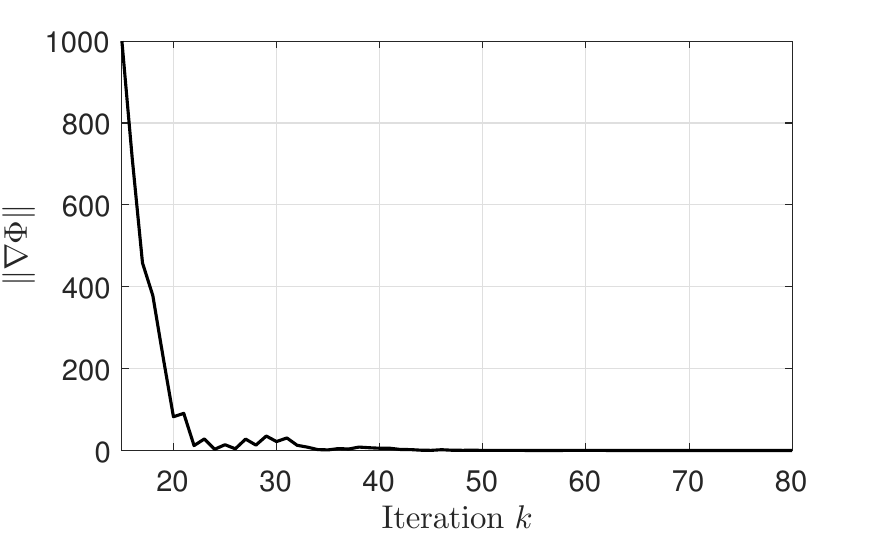}
		\end{minipage}%
	}%
	\subfigure[]{
		\begin{minipage}[t]{0.5\linewidth}
			\centering
			\includegraphics[width=1.8 in, height=1 in]{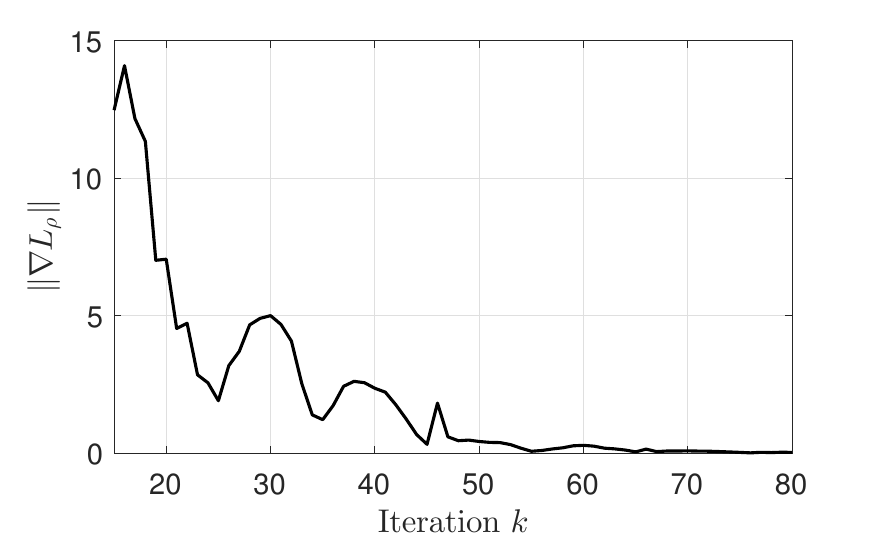}
		\end{minipage}%
	}%
	\centering
	\caption{ (a) the subgradients $\Vert \nabla \Phi_{\beta_k} (\bm{W}^k) \Vert $ w.r.t iteration k.  (b) the subgradients $\Vert \nabla \mathbb{L}_{\rho_{\underline{k}}} (\bm{\bar{X}}^k, \bm{Z}^k, \bm{\lambda}^k, \bm{\mu}^k) \Vert $ w.r.t iteration $k$.} \label{figure3} 
\end{figure}

\begin{figure}[htbp]
	\centering
	\subfigure[]{
		\begin{minipage}[t]{0.5\linewidth}
			\centering
			\includegraphics[width=1.8 in, height=1 in]{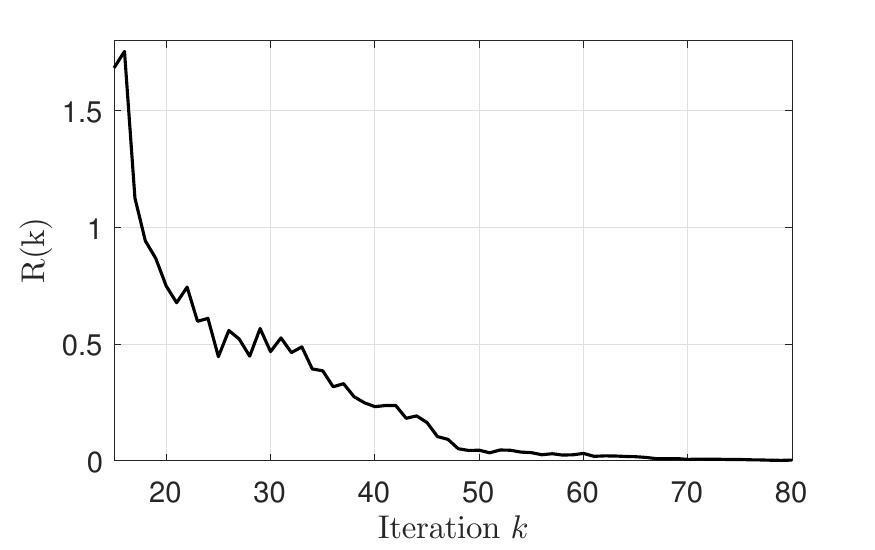}
		\end{minipage}%
	}%
	\subfigure[]{
		\begin{minipage}[t]{0.5\linewidth}
			\centering
			\includegraphics[width=1.8 in, height=1 in]{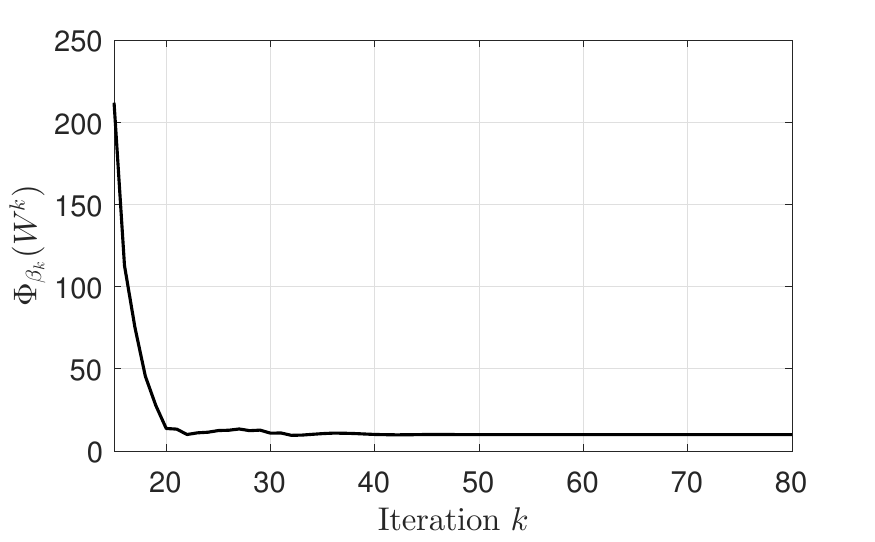}
		\end{minipage}%
	}%
	\centering
	\caption{ (a) the residual error $R(k)$ w.r.t iteration k.  (b) the Lynapunov function $ \Phi_{\beta_k} (\bm{W}^k) $ w.r.t iteration $k$.} \label{figure4} 
\end{figure}



\subsection{Application: Multi-zone HVAC Control}
This section presents the PLDM's application to multi-zone HVAC control, which has raised extensive discussion among the  research communities. The general formulation for designing  a model predictive controller (MPC) to  minimize  the HVAC's  energy cost while guaranteeing thermal comfort   is exhibited in \eqref{HVAC problem} with constraints imposed by  \emph{(i)} zone thermal dynamics (coupled, nonlinear and nonvex) \eqref{thermal dynamics}, \emph{(ii)} comfortable temperature ranges \eqref{thermal comfort},  and \emph{{(iii)}} operation limits of the local variable air  box \eqref{zone air flow rate}. Readers are referred to  \cite{yang2019hvac, radhakrishnan2016token, wang2017distributed} for the detailed  problem formulation and notations.  
\small 
\begin{subequations}
	\begin{alignat}{4} 
&\label{HVAC problem} \min_{m^{zi}_t, T^i_t} J=\sum_{t=0}^{H-1}  c_t \cdot \Big\{  c_p (1-d_r)\sum_{i=1}^I m^{zi}_t (T^o_t-T^c_t)\! \quad \quad \quad \notag\\
&\quad  +c_p d_r \sum_{i=1}^I m^{zi}_t (T^i_t-T^c_t)+  \kappa_f \cdot  (\sum_{i=1}^I {m^{zi}_t})^3
\Big\} \cdot \Delta_t  \tag{40}\\
s.t. ~ 
&\label{thermal dynamics} T^i_{t+1}\!=\! A^{ii} T^i_t\!+\!\sum_{j\in \mathcal{N}_i} A^{ij} T^j_t\!+\!C^{ii} m^{zi}_t (T^i_t\!-\!T^c_t)+D^{ii}_t, \\
&\label{thermal comfort}\quad \underline{T}^i  \leq T^i_t \leq \overline{T}^i, \\ 
&\label{zone air flow rate}\quad \underline{m}^{zi} \leq m^{zi}_t \leq \overline{m}^{zi}, \quad \forall i \in \mathcal{I}, ~t =0, 1, \cdots, H.
\end{alignat}
\end{subequations}
\normalsize

One can see that the control of  the multi-zone HVAC system  requires solving a nonlinear and nonconvex optimization problem with coupled nonlinear constraints. Centralized optimization methods are not scalable or computationally viable and therefore decentralized methods become imperative.  However,  the existing decentralized methods  can not be adapted due to  the non-linearity and non-convexity both  from the objective function  and the constraints.  

In this part, we resort to the proposed PLDM.   We first consider a building with $I=10$ zones and 
 randomly generate a network to denote the thermal couplings  among the different zones. 
The comfortable zone temperature ranges and the maximum zone mass flow rate  are set as $[24, 26] ^oC$ and $\bar{m}^{zi}=0.5 kg/s$. 
The other parameter settings can refer to \cite{yang2019hvac}. 
When the PLDM method is applied, the zone temperature and zone mass flow rates for two randomly picked zones (Zone 1 and 8) are shown in Fig. \ref{figure6} (a) and \ref{figure6} (b). We note that  both the zone temperature and zone mass flow rates  are maintained within the setting ranges.
This implies the feasibility of the solutions. 
Besides,  by observing the residual error w.r.t the iterations  in  Fig. \ref{figure6} (c),  we can imply the fast  convergence rate of the PLDM. 
Further,  to evaluate the sub-optimality of local optima under the PLDM with some initial start points,  we compare the results with a centralized method  in case studies with $I \in \{10, 20\}$ zones.  
In the centralized method,  the optimal solutions can be obtained by solving the non-linear and non-convex problem \eqref{HVAC problem} using Ipopt solver. 
As shown in Fig.  \ref{figure6} (d),  while benchmarking the centralized method (optimal cost is $1$), the
 sub-optimality of the local optima under the  PLDM is  less than $10$\%.  This implies that the PLDM can  approach a satisfactory local optima for non-linear and non-convex problems for some randomly picked initial points. 
In principle, the global optima  for the non-linear problems  can be approached by scattering enough  initial points.
However, some  local optima with  favorable performance and  feasibility guarantee are  generally enough in practice. 
  


\begin{figure}[htbp]
	\centering
	
	\subfigure[]{
		\begin{minipage}[t]{0.5\linewidth}
			\centering
			\includegraphics[width=1.8 in, height=1 in ]{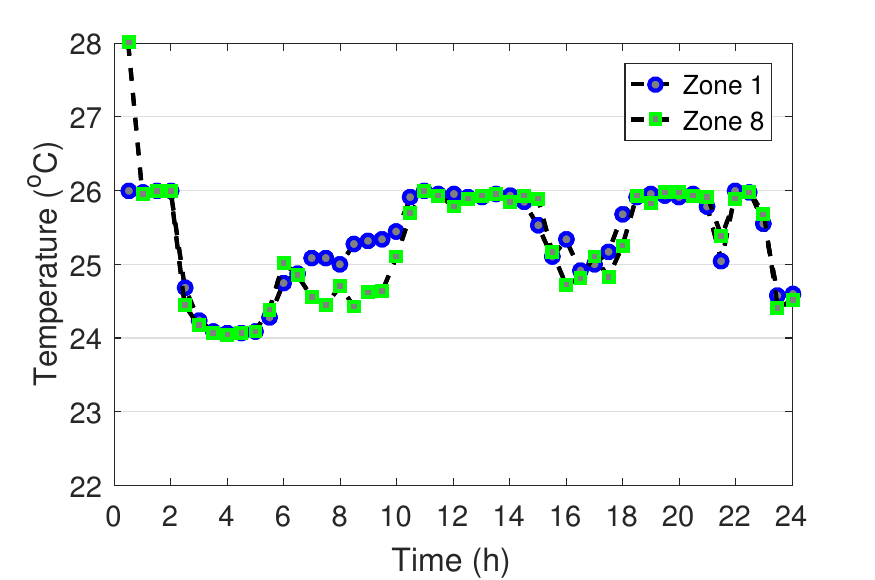}
		\end{minipage}%
	}%
	\subfigure[]{
		\begin{minipage}[t]{0.5\linewidth}
			\centering
			\includegraphics[width=1.8 in, height=1 in]{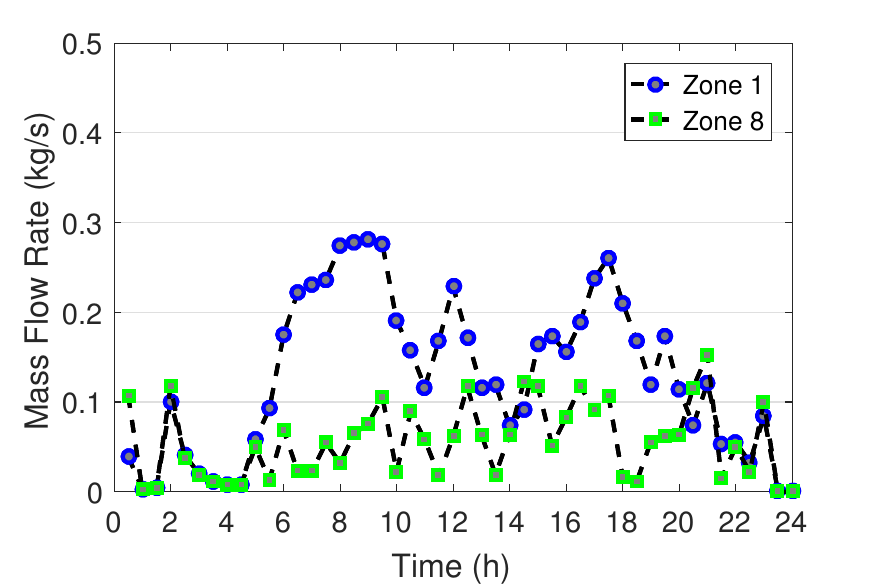}
		\end{minipage}%
	}%

	\subfigure[]{
		\begin{minipage}[t]{0.5\linewidth}
			\centering
			\includegraphics[width=1.8 in, height=1 in]{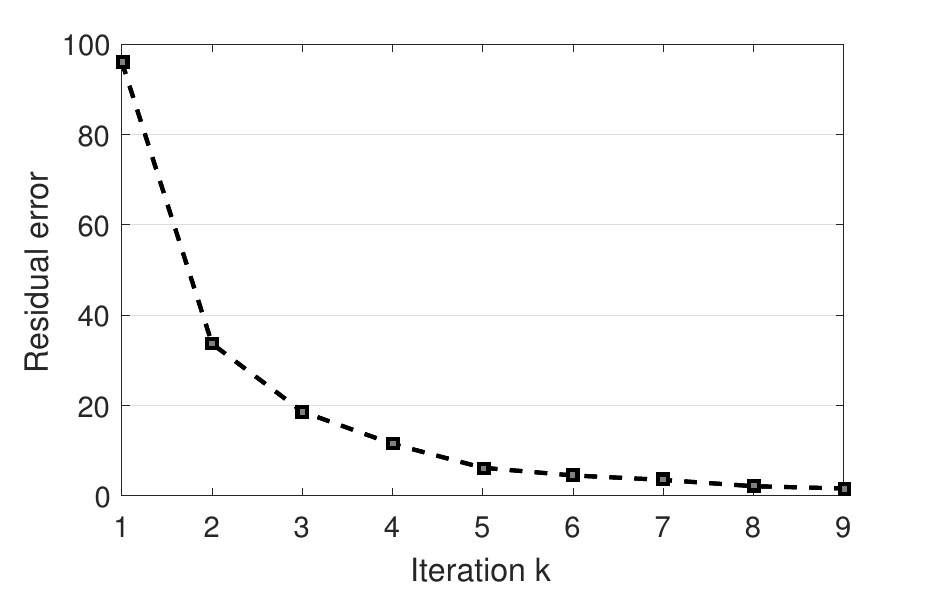}
		\end{minipage}
	}%
	\subfigure[]{
		\begin{minipage}[t]{0.5\linewidth}
			\centering
			\includegraphics[width=1.8 in, height=1 in]{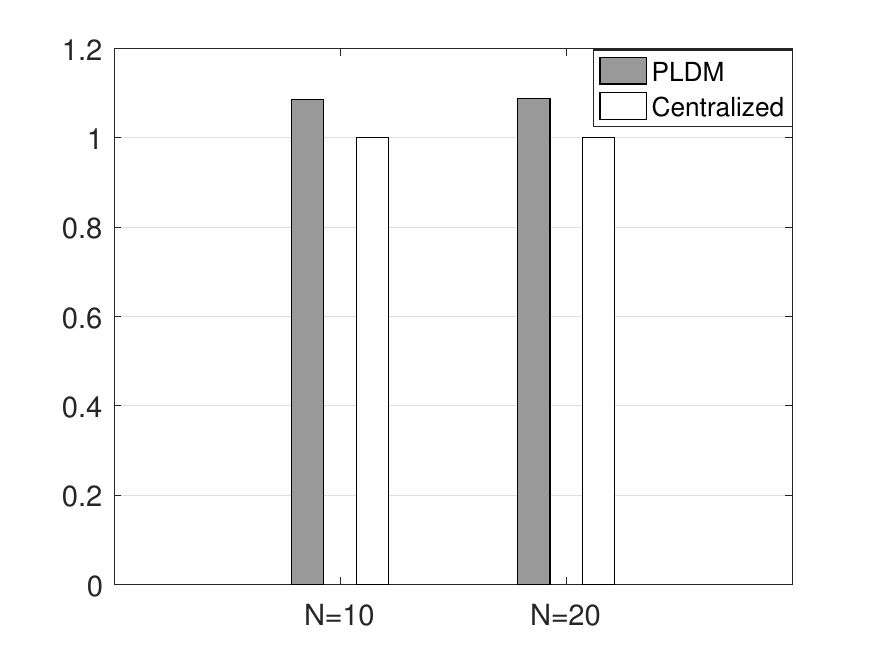}
		\end{minipage}
	}%
	
	\centering
	\caption{ (a)Zone temperature. (b) Zone mass flow rate. (c) Residual error of the constraints. (d) Comparison of PLDM vs. Centralized Method. }
	\label{figure6}
\end{figure}

\section{Conclusion}
This paper investigated decentralized optimization on  a class of  non-convex  problems structured by: {\em{(i)}} nonconvex global objective function (possibly nonsmooth) and {\em{(ii)}} local bounded convex constraints and coupled nonlinear constraints.    Such problems arise from a variety of applications (e.g., smart grid and smart buildings, etc), which demand for efficient and scalable computation. To solve such challenging problems, we proposed a {\emph{Proximal Linearization-based Decentralized Method}} (PLDM) under the augmented Lagrangian-based framework.  Deviating from conventional augmented Lagrangian-based methods which require exact joint  optimization of different decision variable blocks (local optima), the proposed method capitalized on \emph{proximal linearization} technique to update the decision variables at each iteration,  which makes it computationally efficient in the presence of non-linearity. Both the global convergence and the local convergence rate of the PLDM's  to the $\epsilon$-critical points of the problem  were discussed based on the Kurdyka-Łojasiewicz property and some standard assumptions. In addition, the PLDM's performance was illustrated on academic and application examples in which the convergence was observed. Applying PLDM to emerging applications and eliminating consensus variables are the future work to be explored.



\appendices

\section{Proof of \textbf{Proposition} 1-8}

\subsection{Proof of \textbf{Proposition} 1}
\begin{proof}
	We illustrate the proposition by contradiction. We assume $\forall k \in \mathbb{N}$, we have 
	\begin{equation} \label{41}
	\begin{split}
	&\sum_{i=1}^N \Vert h_i(\bm{X}^{k}_i) \Vert>\eta>0
	\end{split}
	\end{equation}
	
In this case, we have $\rho_{k}$ kept increasing, i.e, $ \rho_{k} \rightarrow +\infty$. 
Based on $\bm{\lambda}^{k}_i=\bm{\lambda}^{k-1}_i+\rho_{k-1} h_i(\bm{X}^{k}_i)$, we have
	\begin{equation*}
	\begin{split}
	&\sum_{i=1}^N \Vert h_i(\bm{X}^{k}_i) \Vert \!=\!\frac{1}{\rho_{k-1}} \sum_{i=1}^N \Vert \bm{\lambda}^{k}_i-\bm{\lambda}^{k-1}_i \Vert \leq \frac{\sum_{i=1}^N M_{\bm{\lambda}_i}}{\rho_{k-1}}
	\end{split}
	\end{equation*} 
	($\{\bm{\lambda}^k_i\}_{k \in \mathcal{N}}$ is assumed bounded, see {\em{(A6)}})
	
	The above implies that 
	\begin{equation} \label{42}
	\lim_{k \rightarrow +\infty}\sum_{i=1}^N \Vert h_i(\bm{X}^{k}_i) \Vert \rightarrow 0
	\end{equation}
	
	One may note that \eqref{41} and \eqref{42} contradict with each other. Thus, we conclude that there exist $\underline{k}$ such that  
	\begin{equation*} 
	\begin{split}
	& \sum_{i=1}^N \Vert h_i(\bm{X}^{k}_i) \Vert  \leq \frac{\sum_{i=1}^N M_{\bm{\lambda}_i}}{\rho_{k-1}} \leq \eta, ~~\forall k \geq \underline{k}. \\
	&\{\bm{\bar{X}}^k_i \}_{k \in \mathbb{N}, k \geq  \underline{k}} \in \bm{\mathcal{\bar{X}}}^{\eta}_i~ (i \in \mathcal{N}).
	\end{split}
	\end{equation*}
	Further, based on the adaptive procedure of PLDM, we have
	\begin{equation*}
		 \rho_{k}=\rho_{\underline{k}},~~\forall k \geq \underline{k}.\\
	\end{equation*}

\end{proof}

\subsection{Proof of Proposition 2}
\begin{proof}
	Since $h_i$ is twice differentiable, we have $h_i(\bm{X}_i)$ and $\nabla_{\bm{X}_i} h_i(\bm{X}_i)$ are continuous and  bounded over the bounded set  $\bm{\mathcal{\bar{X}}}^\eta_i$.  In addition, we have $\rho_{k} \leq \rho_{\underline{k}}$ ($\forall  k \in \mathbb{N}$), thus
	\begin{displaymath}
	\big \Vert \rho_{k}  \big( \nabla_{\bm{X}_i} h_i(\bm{X}_i) \big)^T h_i(\bm{X}_i) \big \Vert  \leq \rho_{\underline{k}} C_{h_i}  \leq +\infty ~~(\forall  k \in \mathbb{N} )
	\end{displaymath}
	where $C_{h_i}$ is upper bound of \small $ \Vert ( \nabla_{\bm{X}_i} \!h_i(\bm{X}_i))^T h_i(\bm{X}_i) \Vert$  over $\bm{\mathcal{\bar{X}}}^\eta_i$.  \normalsize 
	
	Besides, we have
	\begin{equation}
	\begin{split}
	&\nabla_{\bm{X}_i} g_i (\bm{\bar{X}}_i, \bm{\lambda}^k_i, \rho_{k})=\nabla_{\bm{X}_i}\tilde{f}_i(\bm{X}_i)+\nabla_{\bm{X}_i} \phi_i(\bm{X}_i)\\
	&\quad +\big(\nabla_{\bm{X}_i} h_i(\bm{X}_i) \big)^T\bm{\lambda}^k_i+\rho_{k} \big(\nabla_{\bm{X}_i} h_i(\bm{X}_i) \big)^T h_i(\bm{X}_i). \\
	&\nabla_{\bm{Y}_i} g_i (\bm{\bar{X}}_i, \bm{\lambda}^k_i, \rho_{k})=2M_i\bm{Y}_i.
	\end{split}
	\end{equation}
	
As $\tilde{f}_i$, $\phi_i$ and $h_i$ are Lipschitz continuous w.r.t $\bm{X}_i$ (or $\bm{\bar{X}}_i$) with constant $L_{f_i}$, $L_{\phi_i}$ and $L_{h_i}$,  $\{\bm{\lambda}^k_i\}_{k \in \mathbb{N}}$ upper bounded by $M_{\bm{\lambda}_i}$ (see \emph{(A6)}),  we conclude that 
$$\Vert \nabla_{\bm{X}_i} g_i (\bm{\bar{X}}_i, \bm{\lambda}^k_i, \rho^k)\Vert \leq L_{g_i}~\textrm{over} ~\bm{\mathcal{\bar{X}}}^\eta_i,~\forall k \in \mathbb{N}.$$
where we have {\small {$L_{g_i}=\max \{L_{f_i}+L_{\phi_i}+ M_{{\lambda}_i}  L_{h_i}+\rho_{\underline{k}} C_{h_i}, 2M_i \}$.}}
 
 That is, $g_i(\bm{\bar{X}}_i, \bm{\lambda}^k_i, \rho_{k})$ ($\forall k \in \mathbb{N}$) is Lipschitz continuous w.r.t $\bm{\bar{X}}_i$  over $\bm{\mathcal{\bar{X}}}^\eta_i$ with  constant $L_{g_i}$.
\end{proof}
 
\subsection{Proof of Proposition 3}
\begin{proof}
	Based on  subproblem (\ref{(6)}), it's straightforward that 
	\begin{eqnarray} \label{(12)}
	\begin{split}
	&\sum_{i=1}^N  \big(\bm{ \mu}^{k}_i \big)^T \big( \bm{A}_i \bm{\bar{X}}^{k}_i\!-\!\bm{E}_i \bm{Z}^{k+1} \big)\!\!+\!\!\sum_{i=1}^N \frac{\rho_{\underline{k}}}{2} \Big \Vert \bm{A}_i \bm{\bar{X}}^{k}_i\!-\!\bm{E}_i \bm{Z}^{k\!+\!1} \Big \Vert^2 \\
	&\leq  \sum_{i=1}^N \big( \bm{\mu}^{k}_i \big)^T \big( \bm{A}_i \bm{\bar{X}}^{k}_i\!-\!\bm{E}_i \bm{Z}^{k} \big)\!\!+\!\!\sum_{i=1}^N \frac{\rho_{\underline{k}}}{2} \Big \Vert  \bm{A}_i \bm{\bar{X}}^{k}_i\!-\!\bm{E}_i \bm{Z}^{k}  \Big \Vert^2 \\
	\end{split}
	\end{eqnarray}
	
	Subproblem (\ref{(7)}) is equivalent to 
	\begin{displaymath}
	\begin{split}
	\bm{\bar{X}}^{k+1}_i  \in  {prox}^{\sigma_i}_{c^k_i} \Big( \bm{\bar{X}}^k_i-\frac{1}{c^k_i} \nabla_{\bm{\bar{X}}_i} g_i(\bm{\bar{X}}^k_i, \bm{\lambda}^k_i, \rho_{\underline{k}})\Big)
	\end{split}
	\end{displaymath}
	where we have $\sigma_i(\bm{\bar{X}}_i)= ( \bm{\mu}_i)^T (\bm{A}_i \bm{\bar{X}}_i-\bm{E}_i \bm{Z}^{k+1} )+ \frac{\rho_{\underline{k}}}{2} \Vert \bm{A}_i \bm{\bar{X}}_i-\bm{E}_i {\bm{Z}}^{k+1}  \Vert^2$.  
	
	Thus, by invoking  \textbf{Lemma} \ref{lemma2}, we have
	\begin{equation} \label{(14)}
	\begin{split}
	&g_i(\bm{\bar{X}}^{k+1}_i, \bm{\lambda}^k_i, \rho_{\underline{k}})+ (\bm{\mu}^{k}_i)^T ( \bm{A}_i \bm{\bar{X}}^{k+1}_i-\bm{E}_i \bm{Z}^{k+1}) \\
	&\quad \quad  +\frac{\rho_{\underline{k}}}{2} \Vert  \bm{A}_i \bm{\bar{X}}^{k+1}_i- \bm{E}_i \bm{Z}^{k+1} \Vert^2\\
	& \leq g_i(\bm{\bar{X}}^{k}_i, \bm{\lambda}^k_i, \rho_{\underline{k}})+ (\bm{\mu}^{k}_i)^T ( \bm{A}_i \bm{X}^{k}_i-\bm{E}_i \bm{Z}^{k+1}) \\
	&\quad  \quad + \frac{\rho_{\underline{k}}}{2} \Vert \bm{A}_i \bm{\bar{X}}^{k}_i- \bm{E}_i \bm{Z}^{k+1} \Vert^2\\
	&\quad \quad \!-\!\frac{1}{2} \big(c_i^k\!-\!L_{g_i}\big) \Vert \bm{\bar{X}}_i^{k+1}-\bm{\bar{X}}^k_i\Vert^2, ~\forall i \in \mathcal{N}. \\
	\end{split}
	\end{equation}
	
	By summing up (\ref{(12)}) and (\ref{(14)}) for   $\forall i \in \mathcal{N}$, we have
	
	\noindent
	\begin{equation} \label{(155)}
	\begin{split}
	&\sum_{i=1}^N g_i(\bm{\bar{X}}^{k+1}_i, \bm{\lambda}^k_i, \rho_{\underline{k}})\!+\!\!\sum_{i=1}^N  ( \bm{\mu}^{k}_i )^T (\bm{A}_i \bm{\bar{X}}^{k+1}_i- \bm{E}_i \bm{Z}^{k+1}) \\
	&\!+\!\sum_{i=1}^N \! \frac{\rho_{\underline{k}}}{2}\! \Vert \bm{A}_i  \bm{\bar{X}}^{k+1}_i\!\!-\!\!\bm{E}_i \bm{Z}^{k+1}  \Vert^2\!\!+\!\!\sum_{i=1}^N\! \big( \bm{\mu}^{k}_i \big)^T\!\! \big(\bm{A}_i  \bm{\bar{X}}^{k}_i\!\!-\!\!\bm{E}_i \bm{Z}^{k+1} \big)\\
	&+\sum_{i=1}^N \frac{\rho_{\underline{k}}}{2}  \Vert \bm{A}_i  \bm{\bar{X}}^{k}_i- \bm{E}_i \bm{Z}^{k+1} \Vert^2\\
	& \leq \sum_{i=1}^N  g_i(\bm{\bar{X}}^{k}_i, \bm{\lambda}^{k}_i, \rho_{\underline{k}} )+\sum_{i=1}^N  ( \bm{\mu}^{k}_i)^T (  \bm{A}_i \bm{\bar{X}}^{k}_i-\bm{E}_i \bm{Z}^{k+1}) \\
	&\!+\!\sum_{i=1}^N \!\! \frac{\rho_{\underline{k}}}{2}\! \Vert  \bm{A}_i \bm{\bar{X}}^{k}_i- \bm{E}_i \bm{Z}^{k+1} \Vert^2\!+\!\sum_{i=1}^N \big( \bm{\mu}^{k}_i \big)^T \big( \bm{A}_i \bm{\bar{X}}^{k}_i- \bm{E}_i \bm{Z}^{k} \big)\\
	&+ \sum_{i=1}^N \frac{\rho_{\underline{k}}}{2} \Vert   \bm{A}_i \bm{\bar{X}}^{k}_i\!-\!\!\bm{E}_i \bm{Z}^{k}  \Vert^2 \\
	&+\sum_{i=1}^N \frac{1}{2} (c^k_i-L_{g_i}) \Vert \bm{\bar{X}}_i^{k+1}-\bm{\bar{X}}^k_i\Vert^2 \\ 
	\end{split}
	\end{equation}
	
	Thus, \textbf{Proposition} \ref{prop4} can be concluded by removing the same terms from both sides of  (\ref{(155)}).
\end{proof}

\subsection{Proof of Proposition 4}
\begin{proof}
	As subproblem (\ref{(7)})  is QP, based on the first-order optimality condition and $\bm{\mu}^{k+1}_i=\bm{\mu}^{k}_i+\rho_{k} ( \bm{A}_i \bm{\bar{X}}^{k+1}_i-\bm{E}_i  \bm{Z}^{k+1} )$, we have
	\begin{equation} \label{(17)}
	\begin{split}
	\nabla_{\bm{\bar{X}}_i} g_i (\bm{\bar{X}}^k_i, &\bm{\lambda}_i^k, \rho_{\underline{k}})\!+ \bm{\mu}^{k+1}_i +c^k_i \big( \bm{\bar{X}}^{k+1}_i\!-\!\bm{\bar{X}}^k_i \big)=0 \\
	\end{split}
	\end{equation}

	Besides, we have 
	\begin{equation}  \label{(18)}
	\begin{split}
	\nabla_{\bm{\bar{X}}_i }  \mathbb{L}_{\rho_{\underline{k}}} &(\bm{\bar{X}}^{k+1},  \bm{Z}^{k+1}, \bm{\lambda}^k, \bm{\mu}^k)\\
	&=\nabla_{\bm{X}_i} g_i(\bm{\bar{X}}^{k+1}_i, \bm{\lambda}^k_i, \rho_{\underline{k}})+\bm{\mu}^{k+1}_i \\
	\end{split}
	\end{equation}

	By combining  (\ref{(17)}) and (\ref{(18)}), we have 
	\begin{equation}  \label{(19)}
	\begin{split}
	&\Vert \nabla_{\bm{\bar{X}}_i } \mathbb{L}_{\rho_{\underline{k}}} (\bm{\bar{X}}^{k+1},  \bm{Z}^{k+1}, \bm{\lambda}^k, \bm{\mu}^k) \Vert\\
	&=\Vert  \nabla_{\bm{X}_i} g_i (\bm{\bar{X}}^{k+1}_i, \bm{\lambda}_i^k, \rho_{\underline{k}})-\nabla_{\bm{\bar{X}}_i} g_i (\bm{\bar{X}}^{k}_i, \bm{\lambda}_i^k, \rho_{\underline{k}})\\
	& \quad -c^k_i \big( \bm{\bar{X}}^{k+1}_i-\bm{\bar{X}}^k_i \big)  \Vert\\
	&\leq  \Vert \nabla_{\bm{\bar{X}}_i} g_i (\bm{\bar{X}}^{k+1}_i, \bm{\lambda}_i^k, \rho_{\underline{k}})-\nabla_{\bm{\bar{X}}_i} g_i (\bm{\bar{X}}^{k}_i, \bm{\lambda}_i^k, \rho_{\underline{k}})  \Vert\\
	&\quad + c^k_i \Vert   \big( \bm{\bar{X}}^{k+1}_i-\bm{\bar{X}}^k_i \big) \Vert \\
	&\leq  \big(L_{g_i}+c_i^k \big) \Vert \bm{\bar{X}}^{k+1}_i-\bm{\bar{X}}^k_i\Vert \\
	\end{split}
	\end{equation}	
	where the last inequality is concluded from \textbf{Proposition} \ref{prop1}.
\end{proof}

\subsection{Proof of Proposition 5}
\begin{proof}
	For notation, we define 
	\begin{equation*}
	\begin{split}
	\Delta^{\bm{\lambda}_i}_k&\!=\! \big( \nabla_{\bm{X}_i} h_i(\bm{X}^{k+1}_i)~\bm{O}_{\bar{n}_i}\big)^T {\bm{\lambda}}^{k+1}_i\!-\!\big( \nabla_{\bm{X}_i} h_i(\bm{X}^{k}_i)~\bm{O}_{\bar{n}_i} \big)^T {\bm{\lambda}}^{k}_i\\
	\Delta^{\bm{\mu}^i }_k&= (\bm{A}_i)^T \bm{\mu}^{k+1}_i-(\bm{A}_i)^T \bm{\mu}^k_i \\
	\end{split}
	\end{equation*}
	
	First, we have
	\begin{equation} \label{50}
	\begin{split}
	\Vert &\Delta^{\bm{\lambda}_i}_k \!+\!\Delta_k^{\bm{\mu}^i}\Vert\!=\!\big\Vert \big(\bm{F}_i(\bm{\bar{X}}^{k+1}_i)\big)^T \bm{\gamma}^{k+1}_i\!-\!\big(\bm{F}_i(\bm{\bar{X}}^k_i) \big)^T \bm{\gamma}^k_i \big\Vert \\
	&=\Vert  \big(\bm{F}_i(\bm{\bar{X}}^{k+1}_i) \big)^T \big(\bm{\gamma}^{k+1}_i - \bm{\gamma}^{k}_i\big)\\
	&\quad  +\big( \bm{F}_i  (\bm{\bar{X}}^{k+1}_i)\big)^T-\big(\bm{F}_i(\bm{\bar{X}}^k_i) \big)^T \bm{\gamma}^{k}_i  \big \Vert \\
	&\geq \big \Vert  \big(\bm{F}_i(\bm{\bar{X}}^{k+1}_i) \big)^T \big(\bm{\gamma}^{k+1}_i - \bm{\gamma}^{k}_i\big) \big \Vert\\
	&\quad  -\big \Vert \big( \bm{F}_i  (\bm{\bar{X}}^{k+1}_i)\big)^T-\big(\bm{F}_i(\bm{\bar{X}}^k_i) \big)^T \bm{\gamma}^{k}_i   \big \Vert\\
	\end{split}
	\end{equation}
	where we have $	\bm{F}_i(\bm{\bar{X}}_i)=
	\left(  
	\begin{array}{ccc}
	\nabla_{\bm{X}_i} h_i(\bm{X}_i) &\!\! & \bm{O}_{n_i} \\
	&\bm{A}_i&\\
	\end{array} 
	\right)	
	$
	and $\bm{\gamma}_i=\big( (\bm{\lambda}_i)^T~(\bm{\mu}_i)^T)^T$ as defined.
	
	According to {\em{(A6)}}, we have  $  \Vert \bm{\gamma}^k_i\Vert^2=\Vert \bm{\lambda}^k_i \Vert^2+\Vert \bm{\mu}^k_i \Vert^2  \leq (M_{\bm{\lambda}_i})^2+(M_{\bm{\mu}_i})^2$ ($\forall k \in \mathbb{N}$). Thus  
	\begin{equation} \label{51}
	\Vert \bm{\gamma}^k_i\Vert \leq M_{\bm{\gamma}_i}~\forall k \in \mathbb{N}
	\end{equation} with
	$M_{\bm{\gamma}_i}=\sqrt{(M_{\bm{\lambda}_i})^2+(M_{\bm{\mu}_i})^2}. $

	Besides, according to {\em{(A1)}}, we have 
	\small
	\begin{eqnarray} \label{52}
	\begin{split}
	 &\Vert \bm{F}_i (\bm{X}^{k+1}_i)\!\!-\!\!\bm{F}_i (\bm{X}^{k}_i) \Vert\!\!=\!\!\Vert \nabla_{\bm{X}_i} h_i(\bm{X}^{k+1}_i)\!\!-\!\! \nabla_{\bm{X}_i} h_i(\bm{X}^{k}_i) \Vert \\
	 & \leq  M_{h_i} \Vert \bm{X}^{k+1}_i-\bm{X}^{k}_i\Vert \leq M_{h_i} \Vert \bm{\bar{X}}^{k+1}_i-\bm{\bar{X}}^{k}_i\Vert, ~\forall k \geq \underline{k}.
	 \end{split}
	 \end{eqnarray} 
	 \normalsize
	 
	 Based on \eqref{51}, \eqref{52} and {\em{(A5)}}, we can derive from \eqref{50} that
	 \small
	 	\begin{equation} \label{53}
	 \begin{split}
	 \Vert &\Delta^{\bm{\lambda}_i}_k \!+\!\Delta_k^{\bm{\mu}^i}\Vert\!\geq \theta \Vert \bm{\gamma}^{k+1}_i\!-\!\bm{\gamma}^k_i\Vert- M_{h_i} M_{\bm{\gamma}_i} \Vert \bm{\bar{X}}^{k+1}_i\!-\!\bm{\bar{X}}^k_i\Vert
	 \end{split}
	 \end{equation}
	\normalsize
	
	Meanwhile, we have
	\begin{displaymath}
	\begin{split}
	\Vert &\Delta^{\bm{\lambda}_i}_k \!+\!\Delta_k^{\bm{\mu}^i}\Vert\!\!=\!\!\big\Vert \big(\bm{F}_i(\bm{\bar{X}}^{k+1}_i)\big)^T \bm{\gamma}^{k+1}_i\!-\!\big(\bm{F}_i(\bm{\bar{X}}^k_i) \big)^T \bm{\gamma}^k_i \big\Vert \\
	&= \Vert \big( \nabla_{\bm{\bar{X}}_i} g_i(\bm{\bar{X}}^{k+1}_i, \bm{\lambda}^{k}_i, \rho_{\underline{k}})+\big(\bm{A}_i \big)^T\bm{ \mu}^{i, k+1} \big) \\
	&\quad  -\big(  \nabla_{\bm{X}_i} g_i(\bm{\bar{X}}^{k}_i, \bm{\lambda}^{k-1}_i, \rho_{\underline{k}} )+\big(\bm{A}_i \big)^T \bm{ \mu}^{i, k} \big) \\
	&\quad  +\big(\nabla_{\bm{X}_i} \tilde{f}_i(\bm{X}^k_i)- \nabla_{\bm{X}_i} \tilde{f}_i(\bm{X}^{k+1}_i)\big)\\
	&\quad +\big(  \nabla_{\bm{X}_i} \phi_i(\bm{X}^k_i)-\nabla_{\bm{X}_i} \phi_i (\bm{X}^{k+1}_i)\big) \Vert \\
		 \end{split}
			\end{displaymath}
		\begin{equation} \label{54}
	\begin{split}
	&\leq \Vert  \nabla_{\bm{\bar{X}}_i } \mathbb{L}_{\rho_{\underline{k}}} (\bm{\bar{X}}^{k+1},  \bm{Z}^{k+1}, \bm{\lambda}^k, \bm{\mu}^k) \Vert \\
	&\quad +\Vert  \nabla_{\bm{\bar{X}}_i } \mathbb{L}_{\rho_{\underline{k}}} (\bm{\bar{X}}^{k},  \bm{Z}^{k}, \bm{\lambda}^{k-1}, \bm{\mu}^{k-1}) \Vert\\
	&\quad +\Vert \nabla_{\bm{X}_i} \tilde{f}_i(\bm{X}^k_i)- \nabla_{\bm{X}_i} \tilde{f}_i(\bm{X}^{k+1}_i)  \Vert \\
	&\quad+ \Vert  \nabla_{\bm{X}_i} \phi_i(\bm{X}^k_i)-\nabla_{\bm{X}_i} \phi_i (\bm{X}^{k+1}_i)  \Vert \\
	&\leq \big(L_{g_i} +c_i^k+L_{f_i}+L_{\phi_i}\big) \Vert \bm{\bar{X}}^{k+1}_i-\bm{\bar{X}}^k_i\Vert \\
	&\quad +\big(L_{g_i}  +c^i_{k\!-\!1} \big) \Vert \bm{\bar{X}}^{k}_i-\bm{\bar{X}}^{k\!-\!1}_i\Vert  \\
	\end{split}
	\end{equation}
	where the last equality is derived from  \textbf{Proposition} \ref{prop5} and {\em{(A2)}}-{\em{(A3)}} .

	By combining \eqref{53} and \eqref{54}, we have
	\begin{equation*}
	\begin{split}
	\Vert \bm{\gamma}^{k+1}_i\!\!-\!\!\bm{\gamma}^k_i\Vert \leq\Omega^i_1 \Vert \bm{\bar{X}}^{k+1}_i\!-\!\bm{\bar{X}}^k_i\Vert\!+\!\Omega^i_2 \Vert \bm{\bar{X}}^{k}_i\!-\!\bm{\bar{X}}^{k-1}_i\Vert  \\
	\end{split}
	\end{equation*}
	where $\Omega^i_1= \big(L_{g_i} +c_i^k+L_{f_i}+L_{\phi_i}+M_{h_i} M_{\bm{\gamma}_i} \big)/\theta$  and 
	$\Omega^i_2=\big(L_{g_i} +c_i^{k\!-\!1})/\theta$ ($i \in \mathcal{N}$).
\end{proof}

\subsection{Proof of Proposition 6}
\begin{proof}
	First, based on $h_i(\bm{X}^{k+1}_i)=\frac{1}{\rho_{k}} (\bm{\lambda}^{k+1}_i-\bm{\lambda}^k_i)$ and $(\bm{A}_i \bm{\bar{X}}^{k+1}_i-\bm{E}_i \bm{Z}^{k+1} )=\frac{1}{\rho_{k}} ( \bm{\mu}^{k+1}_i-\bm{\mu}^{k}_i)$, we have
	\small
	\begin{equation}  \label{(31)}
	\begin{split}
	&\mathbb{L}_{\rho_{\underline{k}}} (\bm{\bar{X}}^{k+1}, \! \bm{Z}^{k+1}, \bm{\lambda}^{k+1}, \bm{\mu}^{k+1})\!-\!\mathbb{L}_{\rho_{\underline{k}}} (\bm{\bar{X}}^{k+1}, \! \bm{Z}^{k+1}, \bm{\lambda}^{k}, \bm{\mu}^{k}) \\
	&= \frac{1}{\rho_{\underline{k}}} \Big (\sum_{i=1}^N \Vert \bm{\lambda}^{k+1}_i-\bm{\lambda}_i^k \Vert^2 +\sum_{i=1}^N\Vert {\bm{\mu}_i}^{k+1}-\bm{\mu}_i^k \Vert^2 \Big)\\
	&= \frac{1}{\rho_{\underline{k}}} \Big (\sum_{i=1}^N \Vert \bm{\gamma}^{k+1}_i-\bm{\gamma}^k_i\Vert^2\Big)\\
	& \leq \sum_{i=1}^N  \Big \{   \frac{2 (\Omega^i_1)^2}{\rho_{\underline{k}}} \Vert \bm{\bar{X}}^{k+1}_i\!\!-\!\!\bm{\bar{X}}^k_i\Vert^2\!\!+\!\!\frac{2 (\Omega^i_2)^2}{\rho_{\underline{k}}} \Vert  \bm{\bar{X}}^k_i\!\!-\!\!\bm{\bar{X}}^{k-1}_i \Vert^2 \Big \}
	\end{split}
	\end{equation}
	\normalsize
	where the last inequality is derived  by invoking  \textbf{\textrm{Proposition}} \ref{prop9} and the Cauchy–Schwarz inequality  $(a+b)^2 \leq 2(a^2+b^2)$.
	
	

	

	Further, by combining (\ref{(11)}) with (\ref{(31)}), we have 
	\begin{equation*}
	\begin{split}
	&\mathbb{L}_{\rho_{\underline{k}}} (\bm{\bar{X}}^k, \bm{Z}^k, \bm{\lambda}^{k}, \bm{\mu}^{k})-\mathbb{L}_{\rho_{\underline{k}}}(\bm{\bar{X}}^{k+1}, \bm{Z}^{k+1}, \bm{\lambda}^{k+1}, \bm{\mu}^{k+1})\\
	& \geq  \sum_{i=1}^N \Big(\frac{1}{2} (c_i^k-L_{g_i})- \frac{2 (\Omega^i_1)^2}{ \rho_{\underline{k}} } \Big) \Vert \bm{\bar{X}}_i^{k+1}-\bm{\bar{X}}^k_i\Vert^2\\
	&+\sum_{i=1}^N (-\frac{2 (\Omega^i_2)^2}{ \rho_{\underline{k}} } ) \Vert \bm{\bar{X}}^{k}_i-\bm{\bar{X}}^{k-1}_i \Vert^2\\
	\end{split}
	\end{equation*}
	
	Thus, we have
	\begin{equation*}
	\begin{split}
	&\Phi_{\beta_k}(\bm{W}^k)\!-\!\Phi_{\beta_{k+1}}(\bm{W}^{k+1})\geq \!\!\sum_{i=1}^N (\beta_k\!-\!\frac{2 (\Omega^i_2)^2}{\rho_{\underline{k}} } ) \Vert \bm{\bar{X}}^{k}_i\!-\!\bm{\bar{X}}^{k-1}_i \Vert^2\\
	&+ \sum_{i=1}^N \Big(\frac{1}{2} (c_i^k\!-\!L_{g_i})\!-\!\! \frac{2 (\Omega^i_1)^2}{\rho_{\underline{k}} }\!-\!\beta_{k+1} \Big) \Vert \bm{\bar{X}}_i^{k+1}\!-\!\bm{\bar{X}}^k_i\Vert^2\\
	\end{split}
	\end{equation*}
	
	\textbf{Proposition} 6 is concluded. 
\end{proof}

\subsection{Proof of Proposition 7}
\begin{proof}
	According to  \textbf{\textrm{Proposition}} \ref{prop11}  and by summing up (\ref{(27)}) for $k \in [\underline{k}, \bar{k}-1]$ ($\bar{k} \in \mathbb{R}$, $\bar{k}>\underline{k}$), we have
	\begin{equation}
	\begin{split}
	&\Phi_{\beta_{\underline{k}}}(\bm{W}^{\underline{k}})-\Phi_{\beta_{\bar{k}}}(\bm{W}^{\bar{k}}) \\
	&\geq \sum_{k=\underline{k}}^{\bar{k}-1} b^k_1   \Vert \bm{\bar{X}}^{k+1}-\bm{\bar{X}}^{k} \Vert^2+\sum_{k=\underline{k}}^K b^k_2   \Vert \bm{\bar{X}}^{k}-\bm{\bar{X}}^{k-1}\Vert^2\\
	& \geq b_1 \sum_{k=\underline{k}}^{\bar{k}-1}   \Vert \bm{\bar{X}}^{k}-\bm{\bar{X}}^{k-1} \Vert^2+b_2 \sum_{k=\underline{k}}^{\bar{k}-1}  \Vert \bm{\bar{X}}^{k+1}-\bm{\bar{X}}^k \Vert^2\\
	\end{split}
	\end{equation}
	where $b_1=\min\limits_{k \in [ \underline{k}, \bar{k}-1] } b^k_1$ and $b_2=\min\limits_{k \in [ \underline{k}, \bar{k}-1] } b^k_2$.
	
	Since $\Phi_{\beta_k} (\bm{W}^k) \geq \inf_{\bm{\bar{X}}, \bm{Z}}  \mathbb{L}_{\rho_{k}} (\bm{\bar{X}}^k, \bm{Z}^k, \bm{\lambda}^k, \bm{\mu}^k) >-\infty$  (see {\em{(A6)}}), let $\bar{k}\rightarrow +\infty$, we have
	\begin{equation} \label{(48)}
	\begin{split}
	& b_1 \sum_{k=\underline{k}}^{+\infty} \Vert \bm{\bar{X}}^{k+1}-\bm{\bar{X}}^k \Vert^2+ b_2 \sum_{k=\underline{k}}^{+\infty} \Vert \bm{\bar{X}}^{k}-\bm{\bar{X}}^{k-1} \Vert^2 \\
	&\leq \Phi_{\beta_{\underline{k}}} (\bm{W}^{\underline{k}})-\lim_{k\in +\infty} \Phi_{\beta_{k}} (\bm{W}^{k}) <+\infty
	\end{split}
	\end{equation}
	
		If  \textbf{Condition} (a) is satisfied, 
		 we have  $b_1, b_2 >0$,  thus
	\begin{equation}
	\begin{split}
	\sum_{k=\underline{k}}^{+\infty} \Vert \bm{\bar{X}}^{k+1}-\bm{\bar{X}}^k\Vert^2 < \infty
	\end{split}
	\end{equation}
	
	This above implies 
	\begin{equation} \label{(49)}
	\begin{split}
	& \lim_{k \rightarrow+\infty }\Vert \bm{\bar{X}}^{k+1}-\bm{\bar{X}}^k\Vert \rightarrow 0\\
	& \lim_{k \rightarrow+\infty }\Vert \bm{\bar{X}}_i^{k+1}-\bm{\bar{X}}_i^k\Vert \rightarrow 0, \forall i \in \mathcal{N}. \\
	\end{split}
	\end{equation}
	
	Further, based on \textbf{Proposition} \ref{prop9}, we have
	\begin{equation}
	\begin{split}
	& \lim_{k \rightarrow  +\infty} \Big \Vert \bm{\gamma}_i^{k+1} -\bm{\gamma}_i^k \Big \Vert \rightarrow 0,\\
	& \lim_{k \rightarrow  +\infty} \Big \Vert \bm{\lambda}_i^{k+1} -\bm{\lambda}_i^k \Big \Vert \rightarrow 0,\\
	& \lim_{k \rightarrow  +\infty} \Big \Vert \bm{\mu}_i^{k+1} -\bm{\mu}_i^k \Big \Vert \rightarrow 0, ~~\forall i \in \mathcal{N}. \\
	\end{split}
	\end{equation}
\end{proof}

\subsection{Proof of Proposition 8}
\begin{proof}
	\emph{(i)} the subgradients of $\Phi_{\beta_{k+1}}  (\bm{W}^{k+1})$ w.r.t  $\bm{\bar{X}}_i$:
	\small
	\begin{eqnarray} \label{63}
	\begin{split}
	& \Vert \nabla_{\bm{\bar{X}}_i} \Phi_{\beta_{k+1}} (\bm{W}^{k+1}) \Vert \\
	& \!=\!\Vert \nabla_{\bm{\bar{X}}_i} \mathbb{L}_{\rho_{\underline{k}}} (\bm{\bar{X}}^{k+1}, \bm{Z}^{k+1}, \bm{\lambda}^{k+1}, \bm{\bm{\mu}^{k+1}})\!+\!2 \beta_{k+1} (\bm{\bar{X}}_i^{k+1}-\bm{\bar{X}}^k_i ) \Vert \\
	&\!=\!\Vert \nabla_{\bm{\bar{X}}_i} \mathbb{L}_{\rho_{\underline{k}}} (\bm{\bar{X}}^{k+1}, \bm{Z}^{k+1}, \bm{\lambda}^{k}, \bm{\bm{\mu}^{k}})\!\!+\!\! \big(\bm{F}_i(\bm{\bar{X}}_i^{k+1}) )^T \big( \bm{\gamma}_i^{k+1}\!\!-\!\!\bm{\gamma}_i^{k} \big) \\
	&\quad +2 \beta_{k+1} ( \bm{\bar{X}}_i^{k+1}-\bm{\bar{X}}^k_i ) \Vert \\
	&\leq (L_{g_i}+c^i_k) \Vert \bm{\bar{X}}^{k+1}_i-\bm{\bar{X}}^k_i\Vert\\
	&\quad +B \Vert \bm{\gamma}_i^{k+1}-\gamma_i^{k} \Vert+2 \beta_{k+1} \Vert \bm{\bar{X}}^{k+1}_i-\bm{\bar{X}}^{k}_i\Vert \\
	&\leq\big( L_{g_i} +c^i_k+ \Omega^i_1 B +2 \beta_{k+1} \big) \Vert \bm{X}^{k+1}_i-\bm{X}^k_i\Vert\\
	&\quad+ \Omega^i_2 B \Vert \bm{X}^{k}_i-\bm{X}^{k-1}_i\Vert\\
	\end{split}
	\end{eqnarray}
	\normalsize
	where the first inequality  is derived from \textbf{Proposition} \ref{prop5} and the definition $B=\sup_{k \geq \underline{k}} \Vert \bm{F}_i(\bm{\bar{X}}_i^{k+1})  \Vert$.  The second  inequality  is derived by invoking \textbf{Proposition} \ref{prop9}.
	
	

	\emph{(ii)} the subgradients of  $\Phi_{\beta_{k+1}}  (\bm{W}^{k+1})$ w.r.t. $\bm{Z}$:
	\small
	\begin{equation}  \label{(38)}
	\begin{split}
	&\nabla_{\bm{Z}} \Phi_{\beta_{k+1}} (\bm{W}^{k+1})  \\
	&=\nabla_{\bm{Z} } \mathbb{L}_{\rho_{k}} (\bm{\bar{X}}^{k+1}, \bm{Z}^{k+1}, \bm{\lambda}^{k+1}, \bm{\bm{\mu}^{k+1}})+\mathcal{N}_{\mathcal{\bm{X}}}(\bm{Z})\\
	&=-\sum_{i=1}^N \big(\bm{E}_i\big)^T \bm{\mu}^{k+1}_i+\mathcal{N}_{\mathcal{X}}(\bm{Z}^{k+1})
	\end{split}
	\end{equation}
	\normalsize

	Based on the first-order optimality condition of  subproblem (\ref{(6)}), we have
	\small
	\begin{equation} \label{(56)}
	\begin{split}
	&-\!\sum_{i=1}^N \big( \bm{E}_i \big)^T \bm{\mu}^{k}_i\!-\!\sum_{i=1}^N  \rho_{k}\big( \bm{E}_i \big)^T \big( \bm{A}_i \bm{\bar{X}}^k_i\!-\!\bm{E}_i \bm{Z}^{k+1}\big)\\
	&+\mathcal{N}_{\mathcal{X}}(\bm{Z}^{k+1})=0
	\end{split}
	\end{equation}
	\normalsize
	
	Further, based on $\bm{\mu}^{k+1}_i\!=\!\bm{\mu}^{k}_i\!+\!\rho_{k} \big( \bm{A}_i \bm{X}_i^{k+1}\!\!-\!\bm{E}_i  \bm{Z}^{k+1} \big)$ and by rearranging  (\ref{(56)}), we have
	\small
	\begin{equation}  \label{(45)}
	\begin{split}
	&-\sum_{i=1}^N\big( \bm{E}_i \big)^T  \bm{\mu}^{k+1}_i+\mathcal{N}_{\mathcal{X}}(\bm{Z}^{k+1})\\
	&=\sum_{i=1}^N  \rho_{\underline{k}} \big( \bm{E}_i \big)^T \big(  \bm{A}_i \bm{X}^{k+1}_i-\bm{A}_i\bm{X}^k_i\big)
	\end{split}
	\end{equation}
	\normalsize
	
	By combining (\ref{(38)}) with (\ref{(45)}), we have
	\small
	\begin{equation}  \label{67}
	\begin{split}
	& \Vert \nabla_{\bm{Z}} \Phi_{\beta_{k+1}} (\bm{W}^{k+1})\Vert\!=\!\Vert \sum_{i=1}^N  \rho_{\underline{k}} \big( \bm{E}_i \big)^T\! \bm{A}_i \big(  \bm{X}^{k+1}_i\!-\!\bm{X}^{k}_i\big) \Vert \\
	&\leq \rho_{\underline{k}} \sum_{i=1}^N \Vert \big(\bm{E}_i\big)^T \bm{A}_i(\bm{X}^{k+1}_i-\bm{X}^k_i)\Vert \\
	& \leq \rho_{\underline{k}} \lambda_{\max}((\bm{E}_i\big)^T \bm{A}_i) \sum_{i=1}^N \Vert\bm{X}^{k+1}_i-\bm{X}^k_i\Vert
	\end{split}
	\end{equation}
	\normalsize
	where $\lambda_{\max}((\bm{E}_i\big)^T \bm{A}_i)$ denotes the maximum eigenvalue of matrix $(\bm{E}_i\big)^T \bm{A}_i$.
	
	\emph{(iii)} the subgradients of  $\Phi_{\beta_{k+1}}  (\bm{W}^{k+1})$ w.r.t. the Lagrangian multipliers $\bm{\lambda}$ and $\bm{\mu}$:
	\begin{equation} \label{(57)}
	\begin{split}
	& \nabla_{\bm{ \lambda}_i } \Phi_{\beta_{k+1}} (\bm{W}^{k+1})\!\!=\!\rho_{\underline{k}} h_i(\bm{X}^{k+1}_i) = \frac{\bm{\lambda}_i^{k+1}\!-\!\bm{\lambda}_i^k}{\rho_{\underline{k}}} \\
	&\nabla_{\bm{\mu}_i } \Phi_{\beta_{k+1}} (\bm{W}^{k+1})\!\!=\!\! \bm{A}_i \bm{\bar{X}}^{k+1}_i\!\!-\!\bm{E}_i  \bm{Z}^{k+1}\!=\!\frac{\bm{\mu}^{k+1}_i-\bm{\mu}^k_i}{\rho_{\underline{k}}} 
	\end{split}
	\end{equation}
	
Based on \textbf{Proposition} \ref{prop9}, we have
	\begin{equation} \label{69}
	\begin{split}
	&\Vert  \nabla_{\bm{ \gamma}_i } \Phi_{\beta_{k+1}}  (\bm{W}^{k+1})\Vert =\Vert \frac{\bm{\gamma}_i^{k+1}\!-\!\bm{\gamma}_i^k}{\rho_{\underline{k}}} \Vert  \\
	&\leq \frac{\Omega^i_1}{ \rho_{\underline{k}}} \Vert \bm{\bar{X}}^{k+1}_i\!\!-\!\!\bm{\bar{X}}^k_i\Vert \!\!+\!\!\frac{\Omega^i_2}{\rho_{\underline{k}}} \Vert  \bm{\bar{X}}^k_i\!\!-\!\!\bm{\bar{X}}^{k-1}_i \Vert 
	\end{split}
	\end{equation}
	
	\emph{(iv)} the subgradients of  $\Phi_{\beta_{k+1}}  (\bm{W}^{k+1})$ w.r.t. $\bm{U}$:
\begin{equation} \label{70}
\begin{split}
& \Vert \nabla_{\bm{U}} \Phi_{\beta_{k+1}} (\bm{W}^{k+1})\Vert\!=\!2\beta_{k+1}\Vert  \bm{\bar{X}}^{k+1}-\bm{\bar{X}}^{k} \Vert
\end{split}
\end{equation}

	By combining \eqref{63}, \eqref{67}, \eqref{69} and \eqref{70}, we have
	\begin{equation}
	\begin{split}
	&\Vert \nabla  \Phi_{\beta_{k+1}} (\bm{W}^{k+1}) \Vert \leq  \sum_{i=1}^N  \big( L_{g_i}+c^k_i +\Omega^i_1 B\\
	&\quad \quad \quad  +4\beta_{k+1}+\rho_{\underline{k}}+\frac{\Omega^i_1}{\rho_{\underline{k}}}\big)\Vert \bm{\bar{X}}^{k+1}_i-\bm{\bar{X}}^k_i\Vert \\
	&\quad + \sum_{i=1}^N \big( \Omega^i_2B +\frac{\Omega^i_2}{\rho_{\underline{k}}}\big)   \Vert \bm{\bar{X}}^{k}_i-\bm{\bar{X}}^{k-1}_i\Vert\\
	\end{split}
	\end{equation}
	\textbf{Proposition}  8 is concluded. 
\end{proof}

\section{Proof of \textbf{Proposition} 9}
\begin{proof}

	By invoking \textbf{Proposition} \ref{prop13}  and the Cauchy–Schwarz inequality  $(\frac{\sum_{i=1}^N x_i}{N})^2 \leq \frac{1}{N} \sum_{i=1}^N (x_i)^2$, we have 
	
	\begin{equation} \label{(50)}
	\begin{split}
	\Big [ \textrm{dist} &\big( \nabla  \Phi_{\beta_{k+1}}( \bm{W}^{k+1}), 0\big)  \Big]^2 \!\leq\!\! 2 N  (b^k_3)^2   \sum_{i=1}^N \Vert \bm{X}^{k+1}_i\!-\!\bm{X}^k_i\Vert^2 \\
	&\quad \quad +2 N (b^k_4)^2 \sum_{i=1}^N  \Vert \bm{\bar{X}}^{k}_i-\bm{\bar{X}}^{k-1}_i\Vert^2 \\
	\end{split}
	\end{equation}
	
	Besides, according to \textbf{\emph{Proposition}} \ref{prop11}, we have
	\begin{equation} \label{(51)}
	\begin{split}
	&\Phi_{\beta_k}(\bm{W}^k)-\Phi_{\beta_{k+1}}(\bm{W}^{k+1}) \\
	& \geq b^k_1  \sum_{i=1}^N \Vert \bm{\bar{X}}_i^{k+1}-\bm{\bar{X}}^k_i\Vert^2+b^k_2 \sum_{i=1}^N \Vert \bm{\bar{X}}^{k}_i-\bm{\bar{X}}^{k-1}_i \Vert^2\\
	\end{split}
	\end{equation}
	
	Thus, by combining (\ref{(50)}), (\ref{(51)}) and \textbf{Condition} (b),  we have 
	\begin{equation} \label{(58)}
	\begin{split}
	\Big [\textrm{dist} \big( \nabla &\Phi_{\beta_{k+1}}( \bm{W}^{k+1}), 0\big) \Big ]^2\\
	& \!\leq \!{2N}{ \nu_k } \big( \Phi_{\beta_k}( \bm{W}^k)-\Phi_{\beta_{k+1}}( \bm{W}^{k+1}) \big)
	\end{split}
	\end{equation}
\end{proof}

\section{Proof of \textbf{Corollary} 1}
\begin{proof}
	If the stepsize $c^k_i$ is selected according to \textbf{Algorithm} \ref{CADMM}, we have
	\begin{equation}\label{(67)}
	\begin{split}
	&g_i(\bm{\bar{X}}_i^{k+1}, \bm{\lambda}^k, \rho_{k})+\alpha \Vert \bm{\bar{X}}^{k+1}_i  -\bm{\bar{X}}_i^k \Vert^2 \\
	& \quad  \leq  g_i(\bm{\bar{X}}^k_i, \bm{\lambda}^k, \rho_{k})+\langle  \nabla_{\bm{\bar{X}}_i} g_i(\bm{\bar{X}}^k_i, \bm{\lambda}^k, \rho_{k} ), \bm{\bar{X}}^{k+1}_i-\bm{\bar{X}}^k_i\rangle\\
	&\quad \quad  \quad \quad \quad +\frac{c_i^k}{2} \big \Vert  \bm{\bar{X}}^{k+1}_i-\bm{\bar{X}}_i^k \big \Vert^2
	\end{split}
	\end{equation}
	
	Based on subproblem (\ref{(7)}), we have
	\begin{equation} \label{(68)}
	\begin{split}
	& (\bm{\mu}^{k}_i)^T ( \bm{A}_i \bm{\bar{X}}^{k+1}_i-\bm{E}_i \bm{Z}^{k+1} )+\frac{\rho_{k}}{2} \Vert \bm{A}_i \bm{\bar{X}}^{k+1}_i\!-\!\bm{E}_i \bm{Z}^{k+1} \Vert^2\!\\
	&g_i(\bm{\bar{X}}_i^{k+1}, \bm{\lambda}^k, \rho_{k})+a \Vert \bm{\bar{X}}^{k+1}_i  -\bm{\bar{X}}_i^k \Vert^2  \\
	&\leq   (\bm{\mu}^{k}_i)^T ( \bm{A}_i \bm{\bar{X}}^{k}_i-\bm{E}_i \bm{Z}^{k+1} )+\frac{\rho_{k}}{2} \Vert  \bm{\bm{A}_i \bar{X}}^{k}_i\!-\!\bm{E}^i \bm{Z}^{k+1}_j \Vert^2 \\
	&+g_i(\bm{\bar{X}}^k_i, \bm{\lambda}^k, \rho_{k})
	\end{split}
	\end{equation}
	
	By combining (\ref{(67)}) with (\ref{(68)}), we have
	\begin{equation}  \label{(69)}
	\begin{split}
	& (\bm{\mu}^{k}_i)^T (\bm{A}_i \bm{\bar{X}}^{k+1}_i-\bm{E}_i \bm{Z}^{k+1} )+\!\frac{\rho_{k}}{2} \Vert  \bm{A}_i \bm{\bar{X}}^{k+1}_i\!-\!\bm{E}^i_j \bm{Z}^{k+1}_j \Vert^2\!\\
	&+\!\langle \nabla_{\bm{\bar{X}}_i} g_i (\bm{\bar{X}}^k_i, \bm{\lambda}_i^k, \rho_{k}), \bm{\bar{X}}^{k+1}_i\!-\!\bm{\bar{X}}^k_i \rangle+\frac{c^k_i}{2} \Vert  \bm{\bar{X}}^{k+1}_i-\bm{\bar{X}}^k_i\Vert^2  \\
	&\leq (\bm{\mu}^{k}_i)^T ( \bm{A}_i \bm{\bar{X}}^{k}_i-\bm{E}_i \bm{Z}^{k+1})+\!\frac{\rho_{k}}{2} \Vert \bm{A}_i \bm{\bar{X}}^{k}_i\!-\!\bm{E}_i \bm{Z}^{k+1} \Vert^2 \\
	\end{split}
	\end{equation}
	
	By summing up (\ref{(69)}) $\forall i \in \mathcal{N}$, we have
	\begin{equation}
	\begin{split}
	&\mathbb{L}_{\rho_{k}}(\bm{\bar{X}}^{k+1}, \bm{Z}^{k+1}, \bm{\lambda}^k, \bm{\mu}^k)+\alpha  \Vert \bm{\bar{X}}^{k+1}  -\bm{\bar{X}}^k \Vert^2\\
	& \quad \quad \leq \mathbb{L}_{\rho_{k}}(\bm{\bar{X}}^{k}, \bm{Z}^{k}, \bm{\lambda}^k, \bm{\mu}^k)
	\end{split}
	\end{equation}
	

	In this case, the convergence of PLDM can be illustrate analogously following  \textbf{Theorem} 1 by replacing $\frac{1}{2}( c^k_i-L_{g_i})$ with $\alpha$. 
	Also,  to guarantee  \textbf{Condition} (a): $\alpha$ should be selected and satisfy  $\alpha \geq \frac{4(\Omega^i_1)^2}{\rho_{\underline{k}}}+4\beta_{k+1}$.
	
\end{proof}



%
\bibliographystyle{ieeetr}
\bibliography{reference}

%








\end{document}